\documentclass[oneside]{amsart}
\usepackage[latin9]{inputenc}
\usepackage{geometry}
\usepackage{xcolor}
\usepackage{amstext}
\usepackage{amsthm}
\usepackage{amssymb}

\makeatletter
\numberwithin{equation}{section}
\numberwithin{figure}{section}
\theoremstyle{plain}
\newtheorem{thm}{\protect\theoremname}[section]
  \theoremstyle{plain}
  \newtheorem{prop}[thm]{\protect\propositionname}
  \theoremstyle{remark}
  \newtheorem{rem}[thm]{\protect\remarkname}
  \theoremstyle{plain}
  \newtheorem{lem}[thm]{\protect\lemmaname}
  \theoremstyle{plain}
  \newtheorem{cor}[thm]{\protect\corollaryname}

\makeatother

  \providecommand{\corollaryname}{Corollary}
  \providecommand{\lemmaname}{Lemma}
  \providecommand{\propositionname}{Proposition}
  \providecommand{\remarkname}{Remark}
\providecommand{\theoremname}{Theorem}

\begin{document}

\title{The Fundamental Solution to One-Dimensional Degenerate Diffusion
Equation, I}

\author{Linan Chen and Ian Weih-Wadman}
\begin{abstract}
In this work we adopt a combination of probabilistic approach and
analytic methods to study the fundamental solutions to variations
of the Wright-Fisher equation in one dimension. To be specific, we
consider a diffusion equation on $\left(0,\infty\right)$ whose diffusion
coefficient vanishes at the boundary 0, equipped with the Cauchy initial
data and the Dirichlet boundary condition. One type of diffusion operator
that has been extensively studied is the one whose diffusion coefficient
vanishes linearly at 0. Our main goal is to extend the study to cases
when the diffusion coefficient has a general order of degeneracy.
We primarily focus on the fundamental solution to such a degenerate
diffusion equation. In particular, we study the regularity properties
of the fundamental solution near 0, and investigate how the order
of degeneracy of the diffusion operator and the Dirichlet boundary
condition jointly affect these properties. We also provide estimates
for the fundamental solution and its derivatives near 0.
\end{abstract}


\address{Department of Mathematics and Statistics, McGill University, 805
Sherbrooke St. West, Montreal, Quebec, H3A 0B9, Canada.}

\email{linan.chen@mcgill.ca, ian.weih-wadman@mail.mcgill.ca}

\thanks{Both authors were partially supported by the NSERC Discovery Grant
(No. 241023).}

\maketitle

\section{Introduction}

In this article we consider the Cauchy initial value problem with
the Dirichlet boundary condition where, given $f\in C_{b}\left(\left(0,\infty\right)\right)$,
we look for $u_{f}\left(x,t\right)\in C^{2,1}\left(\left(0,\infty\right)^{2}\right)$
that satisfies
\begin{equation}
\begin{array}{c}
\partial_{t}u_{f}\left(x,t\right)=a\left(x\right)\partial_{x}^{2}u_{f}\left(x,t\right)+b\left(x\right)\partial_{x}u_{f}\left(x,t\right)\text{ for }\left(x,t\right)\in\left(0,\infty\right)^{2},\\
\lim_{t\searrow0}u_{f}\left(x,t\right)=f\left(x\right)\text{ for }x\in\left(0,\infty\right)\text{ and }\lim_{x\searrow0}u_{f}\left(x,t\right)=0\text{ for }t\in\left(0,\infty\right).
\end{array}\label{eq:general IVP equation}
\end{equation}
Set $L:=a\left(x\right)\partial_{x}^{2}+b\left(x\right)\partial_{x}$.
We are interested in studying the fundamental solution to (\ref{eq:general IVP equation}),
denoted by $p\left(x,y,t\right)$, under the assumption that the diffusion
coefficient $a\left(x\right)$, while being positive on $\left(0,\infty\right)$,
becomes degenerate at the boundary 0, i.e., $\lim_{x\searrow0}a\left(x\right)=0$.
On one hand, because $L$ is not uniformly elliptic on $\left(0,\infty\right)$,
standard techniques in studying fundamental solutions to uniformly
parabolic equations do not apply to $p\left(x,y,t\right)$. On the
other hand, one does expect that $p\left(x,y,t\right)$ exhibits different
properties from a standard heat kernel due to the degeneracy of $a\left(x\right)$,
particularly when $x,y$ are close to the boundary 0. Besides, when
$x,y$ are near 0, $p\left(x,y,t\right)$ also ``feels'' strongly
the influence of the Dirichlet boundary condition, which will force
$p\left(x,y,t\right)$ to vanish at $x=0$. With these considerations
in mind, we want to conduct a study of $p\left(x,y,t\right)$, particularly
to understand how the degeneracy of $a\left(x\right)$ and the Dirichlet
boundary condition together determine the regularity properties of
$p\left(x,y,t\right)$ near 0. In order to carry out this project,
we will impose some conditions on $a\left(x\right)$ and $b\left(x\right)$
that will be made explicit later in this section. 

\subsection{Some previous works on degenerate diffusions. }

Our work is primarily motivated by an earlier work \cite{siamwfeq}
on the well known Wright-Fisher diffusion in the literature of population
genetics:
\begin{equation}
\begin{array}{c}
\partial_{t}u_{f}\left(x,t\right)=x\left(1-x\right)\partial_{x}^{2}u_{f}\left(x,t\right)\text{ for }\left(x,t\right)\in\left(0,1\right)\times\left(0,\infty\right),\\
\lim_{t\searrow0}u_{f}\left(x,t\right)=f\left(x\right)\text{ for }x\in\left(0,1\right),\\
\text{ and }\lim_{x\searrow0}u_{f}\left(x,t\right)=\lim_{x\nearrow1}u_{f}\left(x,t\right)=0\text{ for }t\in\left(0,\infty\right).
\end{array}\label{eq: classical WF eq}
\end{equation}
We will give a brief review of the method and the results in \cite{siamwfeq}.
Set $L_{WF}:=x\left(1-x\right)\partial_{x}^{2}$ and let $p_{WF}\left(x,y,t\right)$
be the fundamental solution to (\ref{eq: classical WF eq}). By studying
the diffusion process associated with $L_{WF}$, the authors of \cite{siamwfeq}
investigate various aspects of $p_{WF}\left(x,y,t\right)$ with a
particular emphasis on its behavior near the boundaries 0 and 1. By
the symmetry of $L_{WF}$ on $\left[0,1\right]$, $p_{WF}\left(x,y,t\right)$
behaves similarly near both boundaries, so it is enough to focus on
one of the boundaries, say, 0. With that in mind, the authors of \cite{siamwfeq}
first consider the operator $L_{0}:=z\partial_{z}^{2}$ on $\left(0,\infty\right)$
and solve the following model equation
\begin{equation}
\begin{array}{c}
\partial_{t}v_{g}\left(z,t\right)=L_{0}v_{g}\left(z,t\right)\text{ for }\left(z,t\right)\in\left(0,\infty\right)^{2},\\
\lim_{t\searrow0}v_{g}\left(z,t\right)=g\left(z\right)\text{ for }z\in\left(0,\infty\right)\text{ and }\lim_{z\searrow0}v_{g}\left(z,t\right)=0\text{ for }t\in\left(0,\infty\right),
\end{array}\label{eq:model equation for WF}
\end{equation}
and find the fundamental solution to (\ref{eq:model equation for WF})
to be
\begin{equation}
q_{0}\left(z,w,t\right)=\frac{z}{t^{2}}e^{-\frac{z+w}{t}}I_{1}\left(\frac{zw}{t^{2}}\right)\text{ for }\left(z,w,t\right)\in\left(0,\infty\right)^{3},\label{eq:q_0 of WF model equation}
\end{equation}
where $I_{1}$ is the modified Bessel function. Next, (\ref{eq: classical WF eq})
is connected to (\ref{eq:model equation for WF}) with a series of
transformations involving localization and perturbation techniques,
which gives rise to a construction (locally near 0) of $p_{WF}\left(x,y,t\right)$
based on $q_{0}\left(z,w,t\right)$. It is shown that for every $\epsilon>0$,
$\left(x,y,t\right)\mapsto y\left(1-y\right)p_{WF}\left(x,y,t\right)$
is smooth with bounded derivatives of all orders on $\left(0,1\right)^{2}\times\left(\epsilon,\infty\right)$.
Furthermore, although $p_{WF}\left(x,y,t\right)$ itself does not
have a closed-form expression, $q_{0}\left(z,w,t\right)$ provides
a sharp estimate for $p_{WF}\left(x,y,t\right)$ locally near 0 when
$t$ is small. To be specific, if
\[
\psi\left(x\right):=\left(\arcsin\sqrt{x}\right)^{2}\text{ for }x\in\left(0,1\right)
\]
and
\[
p_{WF}^{approx.}\left(x,y,t\right):=\frac{q_{0}\left(\psi\left(x\right),\psi\left(y\right),t\right)\psi\left(y\right)}{\sqrt{\psi^{\prime}\left(x\right)\psi^{\prime}\left(y\right)}y\left(1-y\right)}\text{ for }\left(x,y,t\right)\in\left(0,1\right)^{2}\times\left(0,\infty\right),
\]
then for every $0<\alpha<\beta<\gamma$, there exists a constant $C_{\alpha,\beta,\gamma}>0$
such that 
\begin{equation}
\left|\frac{p_{WF}\left(x,y,t\right)}{p_{WF}^{approx.}\left(x,y,t\right)}-1\right|\leq C_{\alpha,\beta,\gamma}t\label{eq:estimate for p_WF}
\end{equation}
for every $t\in(0,1)$ and every $\left(x,y\right)\in\left(0,\alpha\right)^{2}$
with $\left|\arcsin\sqrt{x}-\arcsin\sqrt{y}\right|\leq\gamma-\beta$. 

The estimate (\ref{eq:estimate for p_WF}) is useful for several reasons.
First, it provides the asymptotics of $p_{WF}\left(x,y,t\right)$
in $t$ when $t$ is small and $x,y$ are close to the boundaries.
On one hand, the pioneer work of Kimura \cite{Kimura} gives a construction
of $p_{WF}\left(x,y,t\right)$ as an expansion of the eigenfunctions
of $L_{WF}$; such an expansion describes well the long-term (i.e.,
for large $t$) properties of $p_{WF}\left(x,y,t\right)$ but says
little on its short-term (i.e., for small $t$) properties. On the
other hand, when $\left(x,y\right)$ is away from the boundaries,
one expects that $p_{WF}\left(x,y,t\right)$ behaves similarly as
the fundamental solution to a strictly parabolic equation. So, (\ref{eq:estimate for p_WF})
fills the ``gap'' by providing information on the short-term near-boundary
behaviors of $p_{WF}\left(x,y,t\right)$. Secondly, (\ref{eq:estimate for p_WF})
is more accurate than the general heat kernel estimate. Namely, if
one could overcome the degeneracy of $L_{WF}$ and apply the general
estimates on kernels of parabolic equations (see, e.g., $\mathsection4$
of \cite{PDEStroock}), then one would get that for every $\delta\in(0,1]$,
there exists $C_{\delta}>1$ such that for every $\left(x,y\right)\in\left(0,1\right)^{2}$
and every sufficiently small $t>0$,
\begin{equation}
\frac{C_{\delta}^{-1}}{V_{x,t}}\exp\left(-\frac{d\left(x,y\right)^{2}}{2\left(1-\delta\right)t}\right)\leq p_{WF}\left(x,y,t\right)\leq\frac{C_{\delta}}{V_{x,t}}\exp\left(-\frac{d\left(x,y\right)^{2}}{2\left(1+\delta\right)t}\right),\label{eq:general heat kernel estimate}
\end{equation}
where $d\left(x,y\right)$ is the distance between $x$ and $y$ under
the Riemannian metric on $\left(0,1\right)$ corresponding to $L_{WF}$,
and $V_{x,t}$ is the volume (under the measure induced by the Riemannian
metric) of the ball centered at $x$ with radius $\sqrt{t}$. Although
$\delta>0$ can be arbitrarily small, (\ref{eq:general heat kernel estimate})
does not lead to an approximation of $p_{WF}\left(x,y,t\right)$ whose
ratio with $p_{WF}\left(x,y,t\right)$ can be controlled. So (\ref{eq:estimate for p_WF})
is a strictly sharper estimate than (\ref{eq:general heat kernel estimate})
on $p_{WF}\left(x,y,t\right)$ for small $t$. Moreover, $p_{WF}^{approx.}\left(x,y,t\right)$
has an exact formula in terms of special functions, and the value
of $C_{\alpha,\beta,\gamma}$ is also made explicit in \cite{siamwfeq}.
Hence, (\ref{eq:estimate for p_WF}) is easily accessible in computational
applications of the Wright-Fisher equation.

Independently via an analytic approach, Epstein-Mazzeo \cite{wfeq_Epstein_Mazzeo}
studies the Wright-Fisher equation in a more general setting with
the operator being 
\[
L=x\left(1-x\right)\partial_{x}^{2}+b\left(x\right)\partial_{x}\text{ for }x\in\left(0,1\right),
\]
where $b\left(x\right)$ is smooth on $\left(0,1\right)$ and pointing
inward at the boundaries, i.e., $b\left(0\right)\geq0$ and $b\left(1\right)\leq0$.
Instead of the Dirichlet condition, \cite{wfeq_Epstein_Mazzeo} adopts
the \emph{zero flux} boundary condition:
\[
\lim_{x\searrow0+}x^{b\left(0\right)}\partial_{x}u\left(x\right)=\lim_{x\nearrow1-}\left(1-x\right)^{-b\left(1\right)}\partial_{x}u\left(x\right)=0\text{ for every }t>0.
\]
Under the zero flux condition, the authors of \cite{wfeq_Epstein_Mazzeo}
develop a sharp regularity theory for the solutions, and also derive
the precise asymptotics of the solutions near the boundaries for small
$t$. Epstein and Mazzeo further generalize their work to higher dimensions
by considering operators analogous to $L_{0}$ on\emph{ manifolds
with corners}. To be specific, they assume that for every point $P$
on the boundary of a compact manifold, a neighborhood of $P$ is homeomorphic
to a neighborhood of the origin in $\mathbb{R}_{+}^{n}\times\mathbb{R}^{m}$,
and the operator, referred to as a\emph{ generalized Kimura diffusion
operator}, can be written in local coordinates as:
\[
\begin{split}L= & \sum_{i=1}^{n}\left(x_{i}\partial_{x_{i}}^{2}+b_{i}\left(z\right)\partial_{x_{i}}\right)+\sum_{i,j=1}^{n}x_{i}x_{j}a_{ij}\left(z\right)\partial_{x_{i}}\partial_{x_{j}}\\
 & \hspace{0.4cm}+\sum_{i=1}^{n}\sum_{l=1}^{m}x_{i}c_{il}\left(z\right)\partial_{x_{i}}\partial_{y_{l}}+\sum_{k.l=1}^{m}d_{kl}\left(z\right)\partial_{y_{k}}\partial_{y_{l}}+\sum_{k=1}^{m}e_{k}\left(z\right)\partial_{y_{k}},
\end{split}
\]
where $z=\left(x_{1},\cdots,x_{n},y_{1},\cdots,y_{m}\right)\in\mathbb{R}_{+}^{n}\times\mathbb{R}^{m}$,
the coefficients are smooth and the \emph{weight functions} $b_{i}\left(z\right)\geq0$
when $x_{i}=0$ for each $i$. Epstein and Mazzeo conduct a comprehensive
study on various aspects of such degenerate diffusion equations, including
the Hölder space of the solutions, the maximum principle, the Harnack
inequality, etc.. We refer readers to \cite{EM11,EM13,EM14} for the
details of their results. Related investigations of generalized Kimura
diffusions also include \cite{C0_smooth_parab_Kimura_op,ExiUniq_Markov_Kimura_diff_singulardrift,FeynmanKac_HarnackIneq_deg_diff,tran_prob_deg_diff_manifold}.

Another natural approach in studying the fundamental solution to a
diffusion equation is to consider the corresponding Itô stochastic
integral equation. For our original problem (\ref{eq:general IVP equation}),
we look for a stochastic process $\left\{ X\left(x,t\right):\left(x,t\right)\in\left[0,\infty\right)^{2}\right\} $
that satisfies 
\begin{equation}
X\left(x,t\right)=x+\int_{0}^{t}\sqrt{2a\left(X\left(x,s\right)\right)}dB\left(s\right)+\int_{0}^{t}b\left(X\left(x,s\right)\right)ds\text{ for }\left(x,t\right)\in\left[0,\infty\right)^{2},\label{eq: Ito integral equation for (i)}
\end{equation}
where $\left\{ B\left(s\right):s\geq0\right\} $ is a standard Brownian
motion. Besides, we require that 
\begin{equation}
X\left(x,t\right)\equiv0\text{ for every }x\geq0\text{ and }t\geq\zeta_{0}^{X}\left(x\right),\label{eq:boundary constraint on X(x,t)}
\end{equation}
where $\zeta_{0}^{X}\left(x\right)$ is the hitting time at $0$ of
$X\left(x,t\right)$, i.e., 
\[
\zeta_{0}^{X}\left(x\right):=\inf\left\{ s\geq0:X\left(x,s\right)=0\right\} .
\]
For our purpose, it is sufficient to show that (\ref{eq: Ito integral equation for (i)})
has a weak solution that is unique in law, or equivalently, the martingale
problem associated with $L$ is wellposed; if $\left\{ X\left(x,t\right):\left(x,t\right)\in\left[0,\infty\right)^{2}\right\} $
is the unique solution, then one would expect that for every $\left(x,t\right)\in\left(0,\infty\right)^{2}$,
$y\mapsto p\left(x,y,t\right)$ is given by the probability density
function of $X\left(x,t\right)$ over the set $\left\{ t<\zeta_{0}^{X}\left(x\right)\right\} $.

The existence and the uniqueness of solutions to stochastic integral
equations with degenerate diffusion coefficients have been well studied
(see, e.g., \cite{diff_conti_coeff,singular_stochastic_diff_equa,est_dist_stoch_integral,transf_phasesp_diffproc_removedrift,Ethier76,path_unique_SDE_non-Lip_coeff,deg_SDE_non-Lip_coeff}\textcolor{purple}{{}
}and the references therein). Specifically in the first-order degeneracy
case, when $b\left(x\right)$ is Lipschitz continuous and has at most
linear growth, an application of the theorem of Yamada-Watanabe (\cite{Yamada-Watanabe})
implies the path-wise uniqueness of the solution to (\ref{eq: Ito integral equation for (i)}).
Under the same assumptions on $b\left(x\right),$ Engelbert-Schmidt
\cite{strong_Markov_locmart_solu_1D_SDE} and Cherny \cite{uniq_law_path_uniq_SDE}
complement the Yamada-Watanabe theorem and guarantee that there exists
a path-wise unique strong solution to (\ref{eq: Ito integral equation for (i)}).
In the $d-$dimensional setting, Athreya-Barlow-Bass-Perkins \cite{AthreyaBarlowBassPerkins02}
proves the wellposedness of the martingale problem associated with
the operator 
\[
L=\sum_{i=1}^{d}x_{i}\gamma_{i}\left(x\right)\partial_{x_{i}}^{2}+\sum_{i=1}^{d}b_{i}\left(x\right)\partial_{x_{i}}\text{ for }x\in\mathbb{R}_{+}^{d}
\]
where, for each $i$, $\gamma_{i},b_{i}$ are continuous on $\mathbb{R}_{+}^{d}$,
$\gamma_{i}>0$ on $\mathbb{R}_{+}^{d}$, $b_{i}>0$ on $\partial\mathbb{R}_{+}^{d}$
and $b_{i}$ has at most linear growth. Further, Bass-Perkins \cite{BassPerkins02}
establishes the same conclusion with the condition ``$b_{i}>0$ on
$\partial\mathbb{R}_{+}^{d}$'' relaxed to ``$b_{i}\geq0$ on $\partial\mathbb{R}_{+}^{d}$'',
at the expense of $\gamma_{i}$ and $b_{i}$ being Hölder continuous
on $\mathbb{R}_{+}^{d}$. 

The works mentioned above constitute only a small subset of the rich
literature on degenerate diffusion equations. For example, degenerate
diffusions have also been treated in the context of the measure-valued
process (see, e.g., \cite{measure_value_proc_popul_gen,FV_processes_popul_gen,meas_Markov_proc,resolv_est_FV_uniq_marting,DW_suerproc_meas_diff,wellposed_mart_deg_diff_dynamic_population}),
as well as via the semigroup approach (see, e.g., \cite{C0_semigroup_diffop_ventcel_bdrycond,highly_deg_parabolic_bvp,deg_parabolic_wentzel_bdrycond,analy_semigroup_deg_ellip_operator_nonlinear_Cauchy,deg_evo_eq_reg_semigroup,deg_selfad_evoeq_unitint}). 

\subsection{Our setting.}

All the results we have reviewed above apply to the case when the
diffusion coefficient has first-order degeneracy at the boundary (or
boundaries). \textcolor{black}{The linear degeneracy in the Wright-Fisher
model is inherited from the corresponding discrete models that were
used to model the propagation of a certain allele in population genetics.}\textcolor{purple}{{}
}But from a mathematical point of view, it is natural to consider
the problem when $a\left(x\right)$ does not necessarily degenerate
linearly at the boundary. If $a\left(x\right)$ has a general order
of degeneracy at $0$, one may wonder ``what regularity properties
$p\left(x,y,t\right)$ possesses near 0'', and ``whether it is still
possible to derive a sharp estimate on $p\left(x,y,t\right)$ as in
(\ref{eq:estimate for p_WF})''. The main goal of our work is to
seek answers to these questions. We will restrict ourselves to $a\left(x\right)$
and $b\left(x\right)$ that satisfy the following conditions:\textbf{}\\
\textbf{}\\
\textbf{Condition 1.} \emph{$a\left(x\right)$ is positive and smooth
on $\left(0,\infty\right)$, $a\left(x\right)$ does not vanish too
fast at $0$ in the sense that
\begin{equation}
\lim_{c\searrow0}\int_{c}^{1}\frac{ds}{\sqrt{a\left(s\right)}}<\infty,\label{eq: cond on a(x)}
\end{equation}
and $a\left(x\right)$ does not grow too fast at $\infty$ in the
sense that 
\begin{equation}
\int_{1}^{\infty}\frac{ds}{\sqrt{a\left(s\right)}}=\infty.\label{eq: growth rate of a at infinity}
\end{equation}
}\textbf{Condition 2.}\emph{ $b\left(x\right)$ is smooth on $\left(0,\infty\right)$
such that 
\begin{equation}
\lim_{x\searrow0}\frac{2b\left(x\right)-a^{\prime}\left(x\right)}{4\sqrt{a\left(x\right)}}\int_{0}^{x}\frac{ds}{\sqrt{a\left(s\right)}}\in\left(-\infty,\frac{1}{2}\right),\label{eq: limit existence condition on b}
\end{equation}
i.e., the limit in (\ref{eq: limit existence condition on b}) exists
and is less than $\frac{1}{2}$.}\\

(\ref{eq: cond on a(x)}) and (\ref{eq: limit existence condition on b})
guarantee that the following functions are well defined and smooth
on $\left(0,\infty\right)$: 
\[
\phi:\;x\in\left(0,\infty\right)\mapsto\phi\left(x\right):=\frac{1}{4}\left(\int_{0}^{x}\frac{ds}{\sqrt{a\left(s\right)}}\right)^{2}
\]
and 
\[
d:\;x\in\left(0,\infty\right)\mapsto d\left(x\right):=\frac{2b\left(x\right)-a^{\prime}\left(x\right)}{4\sqrt{a\left(x\right)}}\int_{0}^{x}\frac{ds}{\sqrt{a\left(s\right)}}+\frac{1}{2}-\nu,
\]
where $\nu$ is the constant such that $\lim_{x\searrow0}d\left(x\right)=0$
and $\nu<1$. The constant $\nu$ will play an important role in characterizing
the ``attainability'' type of the boundary 0, which we will discuss
shortly. Our last condition on $a\left(x\right)$ and $b\left(x\right)$
is given in terms of $\phi\left(x\right)$ and $d\left(x\right)$.\textbf{}\\
\textbf{}\\
\textbf{Condition 3.}\emph{ $a\left(x\right)$ and $b\left(x\right)$
are such that if $\phi\left(x\right)$ and $d\left(x\right)$ are
defined as above, then 
\begin{equation}
\sup_{x\in(0,\infty)}\frac{\left|d\left(x\right)\right|}{\sqrt{\phi\left(x\right)}}<\infty\text{ and }\sup_{x\in\left(0,\infty\right)}\frac{\left|d^{\prime}\left(x\right)\right|}{\phi^{\prime}\left(x\right)}<\infty.\label{eq:conditions on b^tilde and z}
\end{equation}
}

Clearly $\phi$ is a strictly increasing function on $\left(0,\infty\right)$,
and by (\ref{eq: growth rate of a at infinity}), $\phi$ is also
surjective on $\left(0,\infty\right)$. Let $\psi:z\in\left(0,\infty\right)\mapsto\psi\left(z\right)\in\left(0,\infty\right)$
be the inverse function of $\phi,$ i.e., $\psi\left(\phi\left(x\right)\right)=x$
for every $x>0$. Then, the conditions in (\ref{eq:conditions on b^tilde and z})
are equivalent to saying that if we set 
\[
\tilde{d}:z\in\left(0,\infty\right)\mapsto\tilde{d}\left(z\right):=d\left(\psi\left(z\right)\right),
\]
then $\tilde{d}\left(z\right)$ is smooth with bounded first derivative
on $\left(0,\infty\right)$ and has at most $\sqrt{z}$ growth rate. 

Inspired by the methods in \cite{siamwfeq}, the approach we adopt
to study $p\left(x,y,t\right)$ is primarily a probabilistic one,
combined with analytic techniques. The general idea is to treat $p\left(x,y,t\right)$
as the transition probability density of the diffusion process associated
with $L$, i.e., the solution $\left\{ X\left(x,t\right):\left(x,t\right)\in\left[0,\infty\right)^{2}\right\} $
to (\ref{eq: Ito integral equation for (i)}) as introduced in $\mathsection1.1$.
However, for general $a\left(x\right)$ and $b\left(x\right)$ that
satisfy Condition 1-3, the standard theory on stochastic integral
equations does not directly imply the existence or the uniqueness
of such a solution. Instead, we first consider the following model
equation
\[
\partial_{t}v\left(z,t\right)=z\partial_{z}^{2}v\left(z,t\right)+\nu\partial_{z}v\left(z,t\right)\text{ for }\left(z,t\right)\in\left(0,\infty\right)^{2},
\]
where $\nu$ is the same constant as in Condition 2. In $\mathsection2$,
we present the complete solution to the model equation including the
exact formula for the fundamental solution, denoted by $q_{\nu}\left(z,w,t\right)$,
and the regularity properties of $q_{\nu}\left(z,w,t\right)$ in spatial
variables near 0. In $\mathsection3$, through a series of transformations,
we connect the original problem (\ref{eq:general IVP equation}) to
the model equation, which leads to a construction of $p\left(x,y,t\right)$
as well as a sharp estimate of $p\left(x,y,t\right)$ for $x,y$ near
0 and for small $t$. $\mathsection4$ delves into the finer structure
of the regularity of $p\left(x,y,t\right)$, where we look for as
accurate as possible estimates on the derivatives of $p\left(x,y,t\right)$
in spatial variables. In $\mathsection5$, we apply our results to
a simple but revealing example: $a\left(x\right)=x^{\alpha}$ for
some $\alpha\in\left[0,2\right]$. For this specific example, it is
clear that $\alpha$ measures the level of degeneracy of the diffusion.
We will see explicitly how $\alpha$ affects various aspects of the
diffusion, including the boundary classification, the regularity of
$p\left(x,y,t\right)$, the hitting distribution at 0 of the underlying
diffusion process, etc..

\subsection{Boundary classification. }

Before getting down to solving (\ref{eq:general IVP equation}), let
us first examine the diffusion operator $L$ in terms of the classification
of the boundary $0$. To this end, we fix an arbitrary $x_{0}>0$
and define, for every $x>0$,
\[
s\left(x\right):=\exp\left(-\int_{x_{0}}^{x}\frac{b\left(u\right)}{a\left(u\right)}du\right),\;S\left(x\right):=\int_{x_{0}}^{x}s\left(u\right)du\text{ and }M\left(x\right):=\int_{x_{0}}^{x}\frac{1}{2a\left(u\right)s\left(u\right)}du.
\]
The functions $S\left(x\right)$ and $M\left(x\right)$ are known
respectively as the \emph{scale measure} and the \emph{speed measure}.
We consider the limits
\[
S_{0}:=-\lim_{x\searrow0}S\left(x\right),\,M_{0}:=-\lim_{x\searrow0}M\left(x\right),
\]
\[
\Sigma:=\lim_{x\searrow0}\int_{x}^{x_{0}}\left(M\left(x_{0}\right)-M\left(u\right)\right)dS\left(u\right)\text{ and }N:=\lim_{x\searrow0}\int_{x}^{x_{0}}\left(S\left(x_{0}\right)-S\left(u\right)\right)dM\left(u\right).
\]
Let $\left\{ X\left(x,t\right):\left(x,t\right)\in\left[0,\infty\right)^{2}\right\} $
be a diffusion process satisfying (\ref{eq: Ito integral equation for (i)}).
For now, let us ignore the constraint (\ref{eq:boundary constraint on X(x,t)})
on $X\left(x,t\right)$ at 0. Then, the behavior of $X\left(x,t\right)$
near $0$ can be classified according to whether each of $S_{0}$,
$M_{0}$, $\Sigma$ or $N$ is finite or infinite (see, e.g., $\mathsection15.6$
of \cite{Karlin_Taylor}). As a demonstration, we examine the boundary
classification of the example with $a\left(x\right)=x^{\alpha}$,
$\alpha\in\left[0,2\right]$, and $b\left(x\right)\equiv0$. \emph{}\\
\emph{}\\
\emph{Case 1.} When $\alpha=2$, $S_{0}<\infty$ and $M_{0}=\Sigma=N=\infty$,
and hence the boundary 0 is a \emph{natural} \emph{boundary} and
\[
\mathbb{P}\left(\zeta_{0}^{X}\left(x\right)=\infty\right)=1\text{ for every }x>0.
\]
In other words, one could omit 0 from the state space without affecting
the behavior of any non-trivial sample path of the diffusion process.\emph{}\\
\emph{}\\
\emph{Case 2.}\textbf{ }When $\alpha\in[1,2)$, $S_{0}<\infty$, $M_{0}=\infty$,
$\Sigma<\infty$ and $N=\infty$, and hence $0$ is an \emph{exit
boundary}. In this case we have that
\[
\mathbb{P}\left(\zeta_{0}^{X}\left(x\right)<\infty\right)>0\text{ for every }x>0
\]
and 
\[
X\left(x,t\right)\equiv0\text{ for all }t\geq\zeta_{0}^{X}\left(x\right).
\]
Once hitting $0$, $X\left(x,t\right)$ is ``stuck'' there, so (\ref{eq:boundary constraint on X(x,t)})
is naturally satisfied and hence the Dirichlet boundary condition
is redundant in this case.\emph{}\\
\emph{}\\
\emph{Case 3.}\textbf{ }When $\alpha\in[0,1)$, $S_{0}$, $M_{0}$,
$\Sigma$ and $N$ are all finite, so $0$ is a \emph{regular boundary}.
We still have that 
\[
\mathbb{P}\left(\zeta_{0}^{X}\left(x\right)<\infty\right)>0\text{ for every }x>0.
\]
In general, upon hitting a regular boundary, the diffusion process
may leave or re-enter the interior of the domain. Therefore, in order
to fully characterize $X\left(x,t\right)$, we need to specify its
behavior at 0, which can range from absorption (i.e., the Dirichlet
boundary condition) to reflection (i.e., the Neumann boundary condition),
or a combination of both (i.e., ``sticky'' boundary condition).
In this case, imposing (\ref{eq:boundary constraint on X(x,t)}) on
$X\left(x,t\right)$ does have an impact on $p\left(x,y,t\right)$. 

\subsection{Some concrete examples. }

Let us review some diffusion equations, for each of which $p\left(x,y,t\right)$
is already known, from the family of equations where $a\left(x\right)=x^{\alpha}$,
$\alpha\in\left[0,2\right]$, and $b\left(x\right)\equiv0$. We will
see how the boundary classification and the imposed boundary condition
affect the derivatives of $p\left(x,y,t\right)$ in $x$ when $x,y$
are near 0.\emph{ }Since in general $p\left(x,y,t\right)$ possesses
certain level of symmetry in $\left(x,y\right)$, one can study the
derivatives of $p\left(x,y,t\right)$ in $y$ via its derivatives
in $x$\emph{.}\\
\emph{}\\
\emph{Example 1. }The easiest example is the case $\alpha=0$, i.e.,
the heat equation on the positive half real line: 
\[
\partial_{t}u\left(x,t\right)=\partial_{x}^{2}u\left(x,t\right)\text{ for }\left(x,t\right)\in\left(0,\infty\right)^{2}.
\]
As we have mentioned above, $0$ is a regular boundary. Upon imposing
the Dirichlet boundary condition, $p\left(x,\cdot,t\right)$ is the
probability density function of $X\left(x,t\right)=\sqrt{2}B\left(t\right)+x$
over the set $\left\{ t<\zeta_{0}^{X}\left(x\right)\right\} $, and
hence
\[
p\left(x,y,t\right)=\frac{1}{\sqrt{\pi t}}e^{-\frac{x^{2}+y^{2}}{4t}}\sinh\left(\frac{xy}{2t}\right)\text{ for every }\left(x,y,t\right)\in\left(0,\infty\right)^{3};
\]
upon imposing the Neumann boundary condition, $p\left(x,\cdot,t\right)$
is the probability density function of $\left|\sqrt{2}B\left(t\right)+x\right|$,
which is 
\[
p\left(x,y,t\right)=\frac{1}{\sqrt{\pi t}}e^{-\frac{x^{2}+y^{2}}{4t}}\cosh\left(\frac{xy}{2t}\right)\text{ for every }\left(x,y,t\right)\in\left(0,\infty\right)^{3}.
\]
We observe that in both cases for fixed $t>0$, the derivatives of
$x\mapsto p\left(x,y,t\right)$ of all orders stay bounded when $x,y$
are near 0 \emph{.}\\
\emph{}\\
\emph{Example 2. }Next, we look at the case when $\alpha=2$, i.e.,
\[
\partial_{t}u\left(x,t\right)=x^{2}\partial_{x}^{2}u\left(x,t\right)\text{ for }\left(x,t\right)\in\left(0,\infty\right)^{2}.
\]
The underlying diffusion process is given by
\[
X\left(x,t\right)=x\exp\left(\sqrt{2}B\left(t\right)-t\right)\text{ for }\left(x,t\right)\in\left[0,\infty\right)^{2}.
\]
It confirms that $0$ is a natural boundary. One can immediately check
that for every $\left(x,y,t\right)\in\left(0,\infty\right)^{3}$
\[
p\left(x,y,t\right):=y^{-2}\cdot\sqrt{\frac{xy}{4\pi t}}\exp\left[-\frac{\left(\ln y-\ln x\right)^{2}}{4t}-\frac{t}{4}\right].
\]
This time we have that for fixed $t>0$, $\partial_{x}p\left(x,y,t\right)$
becomes unbounded as $\left(x,y\right)$ approaches the origin along
the diagonal.\emph{}\\
\emph{}\\
\emph{Example 3. }When $\alpha$ is between 0 and 2, one solvable
case is $\alpha=1$, which is exactly (\ref{eq:model equation for WF}).
This time $0$ is an exit boundary, and we have reviewed in (\ref{eq:q_0 of WF model equation})
that its fundamental solution is 
\[
p\left(x,y,t\right)=\frac{x}{t^{2}}e^{-\frac{x+y}{t}}I_{1}\left(\frac{xy}{t^{2}}\right)
\]
which, for fixed $t>0$, has bounded derivatives in $x$ of all orders
for $x,y$ near 0. \\

Seeing from the three examples above, one may speculate that whether
or not $p\left(x,y,t\right)$ has bounded derivatives in $x$ near
the boundary depends on whether the boundary is \emph{attainable}
(i.e., a regular boundary or an exit boundary) or \emph{unattainable}
(e.g., a natural boundary). However, the next example disproves this
speculation. \\
\emph{}\\
\emph{Example 4. }Consider 
\[
\partial_{t}u\left(x,t\right)=x\partial_{x}^{2}u\left(x,t\right)+\frac{1}{2}\partial_{x}u\left(x,t\right)\text{ for }\left(x,t\right)\in\left(0,\infty\right)^{2}.
\]
One can easily check that $0$ is a regular boundary and the underlying
diffusion process is
\[
X\left(x,t\right)=\left(\sqrt{x}+\frac{1}{\sqrt{2}}B\left(t\right)\right)^{2}\text{ for }\left(x,t\right)\in\left[0,\infty\right)^{2}.
\]
Therefore, without specifying any boundary condition, for every $\left(x,y,t\right)\in\left(0,\infty\right)^{3}$,
\[
p\left(x,y,t\right)=\frac{1}{\sqrt{y}}\frac{1}{\sqrt{\pi t}}e^{-\frac{x+y}{t}}\cosh\left(2\sqrt{\frac{xy}{t^{2}}}\right),
\]
which, for every $t>0$, has bounded derivatives in $x$ of all orders
when $x,y$ are near 0. However, after imposing the Dirichlet boundary
condition,
\[
p\left(x,y,t\right)=\frac{1}{\sqrt{y}}\frac{1}{\sqrt{\pi t}}e^{-\frac{x+y}{t}}\sinh\left(2\sqrt{\frac{xy}{t^{2}}}\right),
\]
whose derivative in $x$ is unbounded as $x$ or $y$ tends to 0.
\\

The examples above indicate that both the boundary classification
and the boundary condition affect the level of regularity of $p\left(x,y,t\right)$,
in terms of the number of bounded derivatives in $x$ near the boundary.
This is one aspect of $p\left(x,y,t\right)$ that will be further
investigated in later sections.

\subsection*{Notations.}

For $c\in\mathbb{R}$, we denote by $\left[c\right]$ the floor of
$c$, i.e., $\left[c\right]:=\sup\left\{ m\in\mathbb{Z}:m\leq c\right\} $.
For $\alpha,\beta\in\mathbb{R}$, we write $\alpha\vee\beta:=\max\left\{ \alpha,\beta\right\} $
and $\alpha\wedge\beta:=\min\left\{ \alpha,\beta\right\} $. For $k\in\mathbb{N}$
and $f\in C^{k}\left(\left(0,\infty\right)\right)$, we set 
\[
C_{k}^{f}:=\max_{j=0,1,\cdots,k}\,\sup_{x\in\left(0,\infty\right)}\left|f^{\left(j\right)}\left(x\right)\right|.
\]
Throughout the article, all the random variables (in particular, stochastic
processes) are assumed to be $\mathbb{R}-$valued and defined on a
generic filtered probability space $\left(\Omega,\mathcal{F},\left\{ \mathcal{F}_{t}:t\geq0\right\} ,\mathbb{P}\right)$.
\\
For an integrable random variable $X$ on $\Omega$ and a set $A\in\mathcal{F}$,
we write $\mathbb{E}\left[X;A\right]:=\int_{A}Xd\mathbb{P}$. \\
Assume that $\left\{ Z\left(t\right):t\geq0\right\} $ is an adapted
stochastic process on $\left(\Omega,\mathcal{F},\left\{ \mathcal{F}_{t}:t\geq0\right\} ,\mathbb{P}\right)$
with almost surely continuous $\mathbb{R}_{+}-$valued sample paths.
Then, for every $x,y\geq0$, we set 
\[
\zeta_{y}^{Z}\left(x\right):=\inf\left\{ t\geq0:Z\left(t\right)=y|Z\left(0\right)=x\right\} ,
\]
i.e., $\zeta_{y}^{Z}\left(x\right)$ is the hitting time at $y$ conditioning
on the process starting at $x$. 

\section{Model Equation}

As we have mentioned in $\mathsection1.2$ that we will construct
the fundamental solution $p\left(x,y,t\right)$ to (\ref{eq:general IVP equation})
based on solving a model equation and applying transformations and
perturbation techniques. In this section we will focus on solving
the following initial/boundary value problem:
\begin{equation}
\begin{array}{c}
\partial_{t}v_{g}\left(z,t\right)=z\partial_{z}^{2}v_{g}\left(z,t\right)+\nu\partial_{z}v_{g}\left(z,t\right)\text{ for }\left(z,t\right)\in\left(0,\infty\right)^{2},\\
\lim_{t\searrow0}v_{g}\left(z,t\right)=g\left(z\right)\text{ for }z\in\left(0,\infty\right)\text{ and }\lim_{z\searrow0}v_{g}\left(z,t\right)=0\text{ for }t\in\left(0,\infty\right),
\end{array}\label{eq:model equation}
\end{equation}
Let us write $L_{\nu}:=z\partial_{z}^{2}+\nu\partial_{z}$ and denote
by $q_{\nu}\left(z,w,t\right)$ the fundamental solution to (\ref{eq:model equation}).
To get started, we first note that the theorem of Yamada-Watanabe
(see, e.g. $\mathsection10$ of \cite{multi_dim_diff_proc}) applies
to this specific case and implies that there exists an almost surely
unique process $\left\{ Y\left(z,t\right):\left(z,t\right)\in\left[0,\infty\right)^{2}\right\} $
satisfying the stochastic integral equation 
\begin{equation}
Y\left(z,t\right)=z+\int_{0}^{t}\sqrt{2\left|Y\left(z,s\right)\right|}dB\left(s\right)+\nu t\text{ for }\left(z,t\right)\in\left[0,\infty\right)^{2}.\label{eq:SDE satisfied by Y(z,t)}
\end{equation}
We further impose the constraint that
\begin{equation}
Y\left(z,t\right)\equiv0\text{ for every }z\geq0\text{ and }t\geq\zeta_{0}^{Y}\left(z\right).\label{eq:boundary conditon on Y(z,t)}
\end{equation}
Let us apply the boundary classification theory to $L_{\nu}$. When
$\nu\leq0$, $0$ is an exit boundary and (\ref{eq:boundary conditon on Y(z,t)})
can be naturally fulfilled; when $0<\nu<1$, 0 is a regular boundary
and hence it is reasonable to achieve the Dirichlet boundary condition
by imposing (\ref{eq:boundary conditon on Y(z,t)}). However, when
$\nu\geq1$, 0 becomes an \emph{entrance boundary}, which is similar
to a natural boundary in the sense that it is unattainable. In this
case, imposing the Dirichlet condition at 0 is ``at odds'' with
the intrinsic behavior of $Y\left(z,t\right)$ near 0; if one wants
to develop any reasonable regularity theory for $q_{\nu}\left(z,w,t\right)$,
one ought to study $L_{\nu}$ under more suitable boundary conditions
such as the zero-flux condition, which will be briefly explained in
Remark \ref{rem:the other formula for q_nu}. In this work, we will
restrict ourselves to the case when 0 is attainable, that is, when
$\nu<1$. 

\subsection{The solution to the model equation.}

Now we get down to determining $q_{\nu}\left(z,w,t\right)$ under
the Dirichlet boundary condition and the assumption that $\nu<1$.
We start with an ``ansatz'' that $q_{\nu}\left(z,w,t\right)$ takes
the form of 
\begin{equation}
q_{\nu}\left(z,w,t\right)=s_{\nu}\left(z,w,t\right)r_{\nu}\left(\frac{zw}{t^{2}}\right),\label{eq:proposed form of q_nu}
\end{equation}
where $s_{\nu}\left(z,w,t\right):=w^{\nu-1}t^{-\nu}e^{-\frac{z+w}{t}}$
and $r_{\nu}:\left(0,\infty\right)\rightarrow\left(0,\infty\right)$
is a function to be determined. If we plug the form (\ref{eq:proposed form of q_nu})
of $q_{\nu}\left(z,w,t\right)$ in the equation in (\ref{eq:model equation}),
then in order for $\left(z,t\right)\mapsto q_{\nu}\left(z,w,t\right)$
to be a solution to (\ref{eq:model equation}) for every $w>0$, we
must have that 
\[
\frac{zw}{t^{2}}r_{\nu}^{\prime\prime}\left(\frac{zw}{t^{2}}\right)+\nu r_{\nu}^{\prime}\left(\frac{zw}{t^{2}}\right)-r_{\nu}\left(\frac{zw}{t^{2}}\right)=0.
\]
In other words, $r_{\nu}$ must be a solution to the ordinary differential
equation 
\begin{equation}
\xi r_{\nu}^{\prime\prime}\left(\xi\right)+\nu r_{\nu}^{\prime}\left(\xi\right)-r_{\nu}\left(\xi\right)=0\label{eq: ode satisfied by r_nu}
\end{equation}
In fact, (\ref{eq: ode satisfied by r_nu}) is closely related to
the equations satisfied by the Bessel functions (see, e.g., $\mathsection3.7$
of \cite{BesselFunctions}). It is not hard to see that 
\[
r_{\nu}\left(\xi\right):=\xi^{1-\nu}\sum_{n=0}^{\infty}\frac{\xi^{n}}{n!\Gamma\left(n+2-\nu\right)}=\xi^{\frac{1-\nu}{2}}I_{1-\nu}(2\sqrt{\xi})
\]
solves (\ref{eq: ode satisfied by r_nu}), where $I_{1-\nu}$ is the
modified Bessel function. Plugging the expression above back into
(\ref{eq:proposed form of q_nu}), we get that for every $\left(z,w,t\right)\in\left(0,\infty\right)^{3}$,
\begin{equation}
q_{\nu}\left(z,w,t\right)=\frac{z^{1-\nu}}{t^{2-\nu}}e^{-\frac{z+w}{t}}\sum_{n=0}^{\infty}\frac{\left(zw\right)^{n}}{t^{2n}n!\Gamma\left(n+2-\nu\right)}=\frac{z^{\frac{1-\nu}{2}}w^{\frac{\nu-1}{2}}}{t}e^{-\frac{z+w}{t}}I_{1-\nu}\left(2\frac{\sqrt{zw}}{t}\right).\label{eq:def of q_nu}
\end{equation}
When $\nu=0$, (\ref{eq:def of q_nu}) coincides with (\ref{eq:q_0 of WF model equation}),
as we have expected.
\begin{prop}
\label{prop: q_nu is the fundamental solution} Assume that $\nu<1$.
Let $q_{\nu}\left(z,w,t\right)$ be defined as in (\ref{eq:def of q_nu}).
Then, $q_{\nu}\left(z,w,t\right)$ is smooth on $\left(0,\infty\right)^{3}$
with 
\[
\lim_{z\searrow0}q_{\nu}\left(z,w,t\right)=0\text{\,for every }\left(w,t\right)\in\left(0,\infty\right)^{2},
\]
and 
\begin{equation}
w^{1-\nu}q_{\nu}\left(z,w,t\right)=z^{1-\nu}q_{\nu}\left(w,z,t\right)\text{ for every }\left(z,w,t\right)\in\left(0,\infty\right)^{3}.\label{eq:symmetry of q_nu}
\end{equation}

In addition, for every $w>0$, $\left(z,t\right)\mapsto q_{\nu}\left(z,w,t\right)$
solves the Kolmogorov backward equation associated with $L_{\nu}$,
i.e.,
\begin{equation}
\left(\partial_{t}-L_{\nu}\right)q_{\nu}\left(z,w,t\right)=0\text{ for }\left(z,t\right)\in\left(0,\infty\right)^{2};\label{eq: q_nu solves backward eq}
\end{equation}
for every $z>0$, $\left(w,t\right)\mapsto q_{\nu}\left(z,w,t\right)$
solves the corresponding Kolmogorov forward equation, i.e.,
\begin{equation}
\left(\partial_{t}-L_{\nu}^{*}\right)q_{\nu}\left(z,w,t\right)=0\text{ for }\left(w,t\right)\in\left(0,\infty\right)^{2},\label{eq: q_nu solves forward eq}
\end{equation}
where $L_{\nu}^{*}:=w\partial_{w}^{2}+\left(2-\nu\right)\partial_{w}$
is the adjoint of $L_{\nu}$. 

Finally, $q_{\nu}\left(z,w,t\right)$ is the fundamental solution
to (\ref{eq:model equation}), and if 
\begin{equation}
v_{g}\left(z,t\right):=\int_{0}^{\infty}q_{\nu}\left(z,w,t\right)g\left(w\right)dw\text{ for }\left(z,t\right)\in\left(0,\infty\right)^{2},\label{eq: def of v_g(z,t)}
\end{equation}
then $v_{g}\left(z,t\right)$ is a smooth solution to (\ref{eq:model equation}).
\end{prop}

\begin{proof}
The smoothness of $q_{\nu}\left(z,w,t\right)$ and the limit of $q_{\nu}\left(z,w,t\right)$
in $z$ as $z\searrow0$, as well as (\ref{eq:symmetry of q_nu}),
(\ref{eq: q_nu solves backward eq}) and (\ref{eq: q_nu solves forward eq}),
all follow from (\ref{eq:def of q_nu}) by direct computations. So
we only need to focus on the last statement. First we observe that
for every $\xi\geq0$,
\[
\sum_{n=0}^{\infty}\frac{\xi^{n}}{n!\Gamma\left(n+2-\nu\right)}\leq\sum_{n=0}^{\infty}\frac{\xi^{n}}{(n!)^{2}}\leq\sum_{n=0}^{\infty}\frac{\left(2\sqrt{\xi}\right)^{2n}}{\left(2n\right)!}\leq e^{2\sqrt{\xi}}
\]
and hence 
\begin{equation}
q_{\nu}\left(z,w,t\right)\leq\frac{z^{1-\nu}}{t^{2-\nu}}e^{-\frac{z+w}{t}}e^{\frac{2\sqrt{zw}}{t}}=\frac{z^{1-\nu}}{t^{2-\nu}}e^{-\frac{\left(\sqrt{z}-\sqrt{w}\right)^{2}}{t}}\text{ for every }\left(z,w,t\right)\in\left(0,\infty\right)^{3}.\label{eq: estimate of q_nu}
\end{equation}
From (\ref{eq: estimate of q_nu}), we immediately derive that for
every $\delta,M>0$, as $t\searrow0$, $\int_{(0,\infty)\backslash\left(z-\delta,z+\delta\right)}q_{\nu}\left(z,w,t\right)dw$
tends to 0 uniformly fast in $z\in(0,M)$. Besides, (\ref{eq: estimate of q_nu})
also guarantees that, for every $\left(z,t\right)\in\left(0,\infty\right)^{2}$,
\begin{equation}
\begin{split}\int_{0}^{\infty}q_{\nu}\left(z,w,t\right)dw & =\frac{z^{1-\nu}}{t^{2-\nu}}e^{-\frac{z}{t}}\sum_{n=0}^{\infty}\frac{z^{n}}{t^{2n}n!\Gamma\left(n+2-\nu\right)}\int_{0}^{\infty}e^{-\frac{w}{t}}w^{n}dw\\
 & =\frac{z^{1-\nu}}{t^{1-\nu}}e^{-\frac{z}{t}}\sum_{n=0}^{\infty}\frac{z^{n}}{t^{n}\Gamma\left(n+2-\nu\right)}=\frac{\gamma(1-\nu,\frac{z}{t})}{\Gamma(1-\nu)},
\end{split}
\label{eq:total mass of q_nu}
\end{equation}
where $\gamma\left(1-\nu,\xi\right):=\int_{0}^{\xi}u^{-\nu}e^{-u}du$,
$\xi>0$, is the incomplete gamma function with 
\[
\lim_{\xi\nearrow\infty}\gamma\left(1-\nu,\xi\right)=\Gamma\left(1-\nu\right).
\]
Thus, as $t\searrow0$, $\int_{0}^{\infty}q_{\nu}\left(z,w,t\right)dw$
tends to 1 uniformly fast for $z\in[\delta,\infty)$ for every $\delta>0$. 

Let $v_{g}\left(z,t\right)$ be defined as in (\ref{eq: def of v_g(z,t)}).
The derivations above leads to $\lim_{z\searrow0}v_{g}\left(z,t\right)=0$
and $\lim_{t\searrow0}v_{g}\left(z,t\right)=g\left(z\right)$ uniformly
in $z$ in any compact set in $\left(0,\infty\right)$. The only thing
left to do is to show that $v_{g}\left(z,t\right)$ is a smooth solution
to the diffusion equation in (\ref{eq:model equation}). Since the
operator $\partial_{t}-L_{\nu}$ is hypoelliptic (see, e.g., $\mathsection7.3$-$\mathsection7.4$
of \cite{PDEStroock}), it is sufficient to show that $v_{g}\left(z,t\right)$
solves the equation in the sense of tempered distributions. To this
end, we take $\varphi$ to be a Schwartz function on $\left(0,\infty\right)$,
and consider 
\[
\left\langle \varphi,v_{g}\left(\cdot,t\right)\right\rangle :=\int_{0}^{\infty}v_{g}\left(z,t\right)\varphi\left(z\right)dz\text{ for }t>0.
\]
Using either of the two expressions in (\ref{eq:def of q_nu}), we
can check that 
\[
\partial_{t}q_{\nu}\left(z,w,t\right)=q_{\nu}\left(z,w,t\right)\frac{z+w}{t^{2}}+\left(\nu-2\right)\frac{1}{t}q_{\nu}\left(z,w,t\right)-\frac{2w}{t^{2}}q_{\nu-1}\left(z,w,t\right).
\]
Then, by (\ref{eq: q_nu solves backward eq}) and (\ref{eq: estimate of q_nu}),
\[
\begin{split}\frac{d}{dt}\left\langle \varphi,v_{g}\left(\cdot,t\right)\right\rangle  & =\int_{0}^{\infty}\left(\int_{0}^{\infty}\partial_{t}q_{\nu}\left(z,w,t\right)\varphi\left(z\right)dz\right)g\left(w\right)dw\\
 & =\int_{0}^{\infty}\left(\int_{0}^{\infty}q_{\nu}\left(z,w,t\right)\left(L_{\nu}^{*}\varphi\right)\left(z\right)dz\right)g\left(w\right)dw\\
 & =\left\langle L_{\nu}^{*}\varphi,v_{g}\left(\cdot,t\right)\right\rangle 
\end{split}
\]
Therefore, $v_{g}\left(z,t\right)$ is a strong solution to the equation
in (\ref{eq:model equation}) and $v_{g}\left(z,t\right)$ is smooth
on $\left(0,\infty\right)^{2}$.
\end{proof}
Next, we will use a martingale argument to prove the uniqueness of
$v_{g}\left(z,t\right)$, where the proof also provides a probabilistic
interpretation of $q_{\nu}\left(z,w,t\right)$.
\begin{prop}
\label{prop:uniqueness of v_g  and CK eq for q_nu}Assume that $\nu<1$.
Given $g\in C_{b}\left(\left(0,\infty\right)\right)$, let $v_{g}\left(z,t\right)$
be defined as in (\ref{eq: def of v_g(z,t)}). Then, 
\begin{equation}
v_{g}\left(z,t\right)=\mathbb{E}\left[g\left(Y\left(z,t\right)\right);t<\zeta_{0}^{Y}\left(z\right)\right]\text{ for every }\left(z,t\right)\in\left(0,\infty\right)^{2}.\label{eq:prob interpretation of v_g}
\end{equation}
and hence $v_{g}(z,t)$ is the unique solution in $C^{2,1}\left(\left(0,\infty\right)^{2}\right)$
to (\ref{eq:model equation}). 

Furthermore, for every Borel set $\Gamma\subseteq[0,\infty)$,
\begin{equation}
\int_{\Gamma}q_{\nu}\left(z,w,t\right)dw=\mathbb{P}\left(Y\left(z,t\right)\in\Gamma,t<\zeta_{0}^{Y}\left(z\right)\right),\label{eq:prob interpretation of q_nu}
\end{equation}
and $q_{\nu}\left(z,w,t\right)$ satisfies the Chapman-Kolmogorov
equation, i.e., for every $z,w>0$ and $t,s>0$,
\begin{equation}
q_{\nu}\left(z,w,t+s\right)=\int_{0}^{\infty}q_{\nu}\left(z,\xi,t\right)q_{\nu}\left(\xi,w,s\right)d\xi.\label{eq:CK equation for q_nu}
\end{equation}
\end{prop}

\begin{proof}
Since $v_{g}\left(z,t\right)$ is smooth on $\left(0,\infty\right)^{2}$
and satisfies the equation in (\ref{eq:model equation}), one can
use Itô's formula to check that for every $\left(z,t\right)\in\left(0,\infty\right)^{2}$,
$\left\{ v_{g}\left(Y\left(z,s\right),t-s\right):s\in\left[0,t\right]\right\} $
is a martingale. Further, by Doob's stopping time theorem, 
\[
\left\{ v_{g}\left(Y\left(z,s\wedge\zeta_{0}^{Y}\left(z\right)\right),t-s\wedge\zeta_{0}^{Y}\left(z\right)\right):s\in\left[0,t\right]\right\} 
\]
is also a martingale. Equating the expectations of the martingale
at $s=0$ and $s=t$ leads to 
\[
v_{g}\left(z,t\right)=\mathbb{E}\left[v_{g}\left(Y\left(z,t\wedge\zeta_{0}^{Y}\left(z\right)\right),t-t\wedge\zeta_{0}^{Y}\left(z\right)\right)\right]=\mathbb{E}\left[g\left(Y\left(z,t\right)\right);t<\zeta_{0}^{Y}\left(z\right)\right],
\]
which is exactly (\ref{eq:prob interpretation of v_g}), and it implies
that $v_{g}\left(z,t\right)$ is the unique solution in $C^{2,1}\left(\left(0,\infty\right)^{2}\right)$
to (\ref{eq:model equation}) since $Y\left(z,t\right)$ is almost
surely unique. (\ref{eq:prob interpretation of q_nu}) follows from
(\ref{eq:prob interpretation of v_g}) because $g\in C_{b}\left(\left(0,\infty\right)\right)$
is arbitrary. 

The uniqueness of the process $\left\{ Y\left(z,t\right):\left(z,t\right)\in\left[0,\infty\right)^{2}\right\} $
also guarantees that the process has the strong Markov property. Therefore,
given any $\varphi\in C_{b}\left(\left(0,\infty\right)\right)$, for
every $z,w>0$ and $t,s>0$,
\[
\begin{split}\int_{0}^{\infty}\varphi\left(w\right)q_{\nu}\left(z,w,t+s\right)dw & =\mathbb{E}\left[\varphi\left(Y\left(z,t+s\right)\right);t+s<\zeta_{0}^{Y}\left(z\right)\right]\\
 & =\mathbb{E}\left[\int_{0}^{\infty}\varphi\left(w\right)q_{\nu}\left(Y\left(z,t\right),w,s\right)dw;t<\zeta_{0}^{Y}\left(z\right)\right]\\
 & =\int_{0}^{\infty}\int_{0}^{\infty}\varphi\left(w\right)q_{\nu}\left(u,w,s\right)q_{\nu}\left(z,u,t\right)dudw,
\end{split}
\]
which leads to (\ref{eq:CK equation for q_nu}).
\end{proof}
\begin{rem}
\label{rem:the other formula for q_nu} While determining the formula
of $q_{\nu}\left(z,w,t\right)$ earlier, we obtained the ordinary
differential equation (\ref{eq: ode satisfied by r_nu}) which led
to $r_{\nu}\left(\xi\right)$. However, it is easy to see that
\[
r_{\nu}^{*}\left(\xi\right):=\sum_{n=0}^{\infty}\frac{\xi^{n}}{n!\Gamma\left(\nu+n\right)}=\xi^{\frac{1-\nu}{2}}I_{\nu-1}\left(2\sqrt{\xi}\right)\text{ for }\xi>0
\]
is also a solution to (\ref{eq: ode satisfied by r_nu}), and hence
if we define, for every $\left(z,w,t\right)\in\left(0,\infty\right)^{3}$,
\[
q_{\nu}^{*}\left(z,w,t\right):=\frac{w^{\nu-1}}{t^{\nu}}e^{-\frac{z+w}{t}}\sum_{n=0}^{\infty}\frac{\left(zw\right)^{n}}{t^{2n}n!\Gamma\left(\nu+n\right)}=\frac{w^{\frac{\nu-1}{2}}z^{-\frac{1-\nu}{2}}}{t}e^{-\frac{z+w}{t}}I_{\nu-1}\left(\frac{2\sqrt{zw}}{t}\right),
\]
then $\left(z,t\right)\mapsto q_{\nu}^{*}\left(z,w,t\right)$ also
solves the equation in (\ref{eq:model equation}) but does not satisfy
the Dirichlet boundary condition. In fact, (\ref{eq:model equation})
is exactly the model equation treated in \cite{wfeq_Epstein_Mazzeo}
except that there the constant $\nu$ is assumed to be non-negative
and the boundary condition is the zero flux condition:
\[
\lim_{z\searrow0}z^{\nu}\partial_{z}v\left(z,t\right)=0\text{ for every }t>0,
\]
for which the fundamental solution is found to be $q_{\nu}^{*}\left(z,w,t\right)$.
Obviously $q_{\nu}\left(z,w,t\right)$ and $q_{\nu}^{*}\left(z,w,t\right)$
have distinct regularity properties near 0. For example, for every
$\left(w,t\right)\in\left(0,\infty\right)^{2}$, $z\mapsto q_{\nu}^{*}\left(z,w,t\right)$
is analytic near 0 for any value of $\nu$, while, as it will be made
explicit in the next subsection, $z\mapsto q_{\nu}\left(z,w,t\right)$
has only finitely many orders of bounded derivatives near 0 for non-integer
$\nu$. 
\end{rem}

\subsection{Derivative of solutions to the model equation.}

Since we have obtained the explicit formulas of $q_{\nu}\left(z,w,t\right)$
and $v_{g}\left(z,t\right)$, we can take a closer look at their derivatives
in $z$. Let us first slightly extend the definition in (\ref{eq:def of q_nu}).
Although we will only consider the case when $\nu<1$ in this article,
when we treat the derivatives of $q_{\nu}\left(z,w,t\right)$, we
will encounter functions of the same type but corresponding to larger
values of $\nu$. To be convenient, for every constant $\sigma\in\mathbb{R}\backslash\left\{ 2,3,\cdots\right\} $,
we continue defining $q_{\sigma}\left(z,w,t\right)$ as 
\begin{equation}
q_{\sigma}\left(z,w,t\right):=\frac{z^{1-\sigma}}{t^{2-\sigma}}e^{-\frac{z+w}{t}}\sum_{n=0}^{\infty}\frac{\left(zw\right)^{n}}{t^{2n}n!\Gamma\left(n+2-\sigma\right)}=\frac{z^{\frac{1-\sigma}{2}}w^{\frac{\sigma-1}{2}}}{t}e^{-\frac{z+w}{t}}I_{1-\sigma}\left(2\frac{\sqrt{zw}}{t}\right).\label{eq:def of q_sigma}
\end{equation}
Note that even if $\sigma\geq1$, neither of the two expressions in
(\ref{eq:def of q_sigma}) causes trouble, because $\Gamma\left(2-\sigma\right)$
and $I_{1-\sigma}$ are well defined for all $\sigma\in\mathbb{R}\backslash\left\{ 2,3,\cdots\right\} $.
We have the following technical lemma on $q_{\sigma}\left(z,w,t\right)$. 
\begin{lem}
\label{lem:technical lemma}Let $q_{\sigma}\left(z,w,t\right)$ be
defined as in (\ref{eq:def of q_sigma}) for every $\sigma\in\mathbb{R}\backslash\left\{ 2,3,\cdots\right\} $.\\
\\
(i) If $\sigma\in[1,\infty)\backslash\left\{ 2,3,\cdots\right\} $,
then for every $\left(z,w,t\right)\in\left(0,\infty\right)^{3}$,
\begin{equation}
\left|q_{\sigma}\left(z,w,t\right)\right|\leq\frac{z^{1-\sigma}}{t^{2-\sigma}}\left(\frac{zw}{t^{2}}\vee1\right)^{\left[\sigma-2\right]+1}\frac{\left(\left[\sigma-2\right]+2\right)!}{\Gamma\left(4-\sigma+\left[\sigma-2\right]\right)}e^{-\frac{z+w}{t}}+\frac{z^{3-\sigma+\left[\sigma-2\right]}}{t^{4-\sigma+\left[\sigma-2\right]}}e^{-\frac{\left(\sqrt{z}-\sqrt{w}\right)^{2}}{t}}.\label{eq:estimate of q_nu for nu>=00003D1}
\end{equation}
(ii) For every $\nu\in\left(-\infty,1\right)\backslash\mathbb{Z}$,
every $k\in\mathbb{N}$ and every $\left(z,w,t\right)\in\left(0,\infty\right)^{3}$,
\begin{equation}
\partial_{z}^{k}q_{\nu}\left(z,w,t\right)=\frac{1}{t^{k}}\sum_{j=0}^{k}\binom{k}{j}\left(-1\right)^{k-j}q_{\nu+j}\left(z,w,t\right)\label{eq:(dz)^k of q_nu expressed as sum of q_nu+j}
\end{equation}
and 

\begin{equation}
\partial_{z}^{k}q_{\nu}\left(z,w,t\right)=\left(-1\right)^{k}\partial_{w}^{k}q_{\nu+k}\left(z,w,t\right).\label{eq:(dz)^k q_nu =00003D (dw)^k q_(nu+k)}
\end{equation}
(iii) For every $N\in\mathbb{N}$, every $k\in\mathbb{N}$ and every
$\left(z,w,t\right)\in\left(0,\infty\right)^{3}$,
\begin{equation}
\begin{split}\partial_{z}^{k}q_{-N}\left(z,w,t\right) & =\frac{1}{t^{k}}\sum_{j=0}^{\left(N+1\right)\wedge k}\binom{k}{j}\left(-1\right)^{k-j}q_{-N+j}\left(z,w,t\right)\\
 & \qquad+\frac{1}{t^{k}}\sum_{j=\left(N+2\right)\wedge k}^{k}\binom{k}{j}\left(-1\right)^{k-j}q_{2+N-j}\left(w,z,t\right)
\end{split}
\label{eq:(dz)^k of q_(-N)}
\end{equation}
and 
\begin{equation}
\partial_{z}^{k}q_{-N}\left(z,w,t\right)=\begin{cases}
\left(-1\right)^{k}\partial_{w}^{k}q_{-N+k}\left(z,w,t\right), & \text{if }k\in\left\{ 0,1,2,\cdots,N+1\right\} ,\\
\left(-1\right)^{k}\partial_{w}^{k}q_{2+N-k}\left(w,z,t\right), & \text{if }k\geq N+2.
\end{cases}\label{eq:(dz)^k q_(-N) =00003D (dw)^k q_(N+2-k)}
\end{equation}
\end{lem}

\begin{proof}
We review that for $\alpha>0$, $\Gamma\left(\alpha\right)$ is defined
by the integral expression as 
\[
\Gamma\left(\alpha\right):=\int_{0}^{\infty}t^{\alpha-1}e^{-t}dt,
\]
and for $\alpha<0$, $\alpha\notin\mathbb{Z}$, $\Gamma\left(\alpha\right)$
is determined by the relation that
\[
\Gamma\left(\alpha\right):=\frac{\Gamma\left(\alpha+\left[-\alpha\right]+1\right)}{\alpha\left(\alpha+1\right)\left(\alpha+2\right)\cdots\left(\alpha+\left[-\alpha\right]\right)}.
\]
For $\sigma\in[1,\infty)\backslash\left\{ 2,3,\cdots\right\} $, to
prove (\ref{eq:estimate of q_nu for nu>=00003D1}), it suffices to
rewrite the series in (\ref{eq:def of q_sigma}) as 
\[
\begin{split} & \sum_{n=0}^{1+\left[\sigma-2\right]}\frac{\left(zw\right)^{n}\left(n+2-\sigma\right)\left(n+3-\sigma\right)\cdots\left(\left[\sigma-2\right]+3-\sigma\right)}{t^{2n}n!\Gamma\left(4-\sigma+\left[\sigma-2\right]\right)}\\
 & \hspace{1cm}+\left(\frac{zw}{t^{2}}\right)^{\left[\sigma-2\right]+2}\sum_{m=0}^{\infty}\frac{\left(zw\right)^{m}}{t^{2m}\left(m+2+\left[\sigma-2\right]\right)!\Gamma\left(m+4+\left[\sigma-2\right]-\sigma\right)}\\
\leq & \frac{\left(\left[\sigma-2\right]+1\right)!}{\Gamma\left(4-\sigma+\left[\sigma-2\right]\right)}\sum_{n=0}^{1+\left[\sigma-2\right]}\frac{\left(zw\right)^{n}}{t^{2n}n!}+\left(\frac{zw}{t^{2}}\right)^{\left[\sigma-2\right]+2}\sum_{m=0}^{\infty}\frac{\left(zw\right)^{m}}{t^{2m}\left(m!\right)^{2}}\\
\leq & \frac{\left(\left[\sigma-2\right]+2\right)!}{\Gamma\left(4-\sigma+\left[\sigma-2\right]\right)}\left(\frac{zw}{t^{2}}\vee1\right)^{1+\left[\sigma-2\right]}+\left(\frac{zw}{t^{2}}\right)^{\left[\sigma-2\right]+2}e^{\frac{2\sqrt{zw}}{t}}.
\end{split}
\]

Now let $\sigma\in\mathbb{R}\backslash\left\{ 2,3,\cdots\right\} $
and $k\in\mathbb{N}$. Using either of the two expressions in (\ref{eq:def of q_sigma}),
we can check by direct computations that 
\[
\partial_{z}q_{\sigma}\left(z,w,t\right)=-\frac{1}{t}q_{\sigma}\left(z,w,t\right)+\frac{1}{t}q_{\sigma+1}\left(z,w,t\right)=-\partial_{w}q_{\sigma+1}\left(z,w,t\right)\text{ for every }\left(z,w,t\right)\in\left(0,\infty\right)^{3}.
\]
We will use induction to prove (\ref{eq:(dz)^k of q_nu expressed as sum of q_nu+j})
and (\ref{eq:(dz)^k q_nu =00003D (dw)^k q_(nu+k)}). There is nothing
to be done for $k=0$. Assume that the two formulas are correct for
some $k\in\mathbb{N}$. Then, 
\[
\begin{split}\partial_{z}^{k+1}q_{\nu}\left(z,w,t\right) & =\frac{1}{t^{k}}\sum_{j=0}^{k}\binom{k}{j}\left(-1\right)^{k-j}\partial_{z}q_{\nu+j}\left(z,w,t\right)\\
 & =\frac{1}{t^{k+1}}\sum_{j=0}^{k}\binom{k}{j}\left(-1\right)^{k+1-j}\left(q_{\nu+j}\left(z,w,t\right)-q_{\nu+j+1}\left(z,w,t\right)\right)\\
 & =\frac{1}{t^{k+1}}\sum_{j=0}^{k+1}\binom{k+1}{j}\left(-1\right)^{k+1-j}q_{\nu+j}\left(z,w,t\right)
\end{split}
\]
and 
\[
\partial_{z}^{k+1}q_{\nu}\left(z,w,t\right)=\left(-1\right)^{k}\partial_{w}^{k}\partial_{z}q_{\nu+k}\left(z,w,t\right)=\left(-1\right)^{k+1}\partial_{w}^{k+1}q_{\nu+k+1}\left(z,w,t\right).
\]
Hence, (\ref{eq:(dz)^k of q_nu expressed as sum of q_nu+j}) and (\ref{eq:(dz)^k q_nu =00003D (dw)^k q_(nu+k)})
also hold for $k+1$.

Now we move on to (\ref{eq:(dz)^k of q_(-N)}). If $k\leq N+1$, then
$q_{-N+k}\left(z,w,t\right)$ is still defined by (\ref{eq:def of q_sigma}).
It is clear that the arguments above that were used to prove (\ref{eq:(dz)^k of q_nu expressed as sum of q_nu+j})
and (\ref{eq:(dz)^k q_nu =00003D (dw)^k q_(nu+k)}) still apply, and
(\ref{eq:(dz)^k of q_(-N)}) and (\ref{eq:(dz)^k q_(-N) =00003D (dw)^k q_(N+2-k)})
coincide with (\ref{eq:(dz)^k of q_nu expressed as sum of q_nu+j})
and (\ref{eq:(dz)^k q_nu =00003D (dw)^k q_(nu+k)}) respectively in
this case. Assuming $k\geq N+2$, we will complete the proof by induction
on $k$ again. First, we verify by direct computations that 
\begin{equation}
\partial_{z}q_{1}\left(z,w,t\right)=-\frac{1}{t}q_{1}\left(z,w,t\right)+\frac{1}{t}e^{-\frac{z+w}{t}}\sum_{n=1}^{\infty}\frac{w^{n}z^{n-1}}{t^{2n}\left(n-1\right)!n!}=-\frac{1}{t}q_{1}\left(z,w,t\right)+\frac{1}{t}q_{0}\left(w,z,t\right),\label{eq:(dz)q_1}
\end{equation}
so, by (\ref{eq:(dz)^k of q_nu expressed as sum of q_nu+j}), 
\begin{align*}
 & \partial_{z}^{N+2}q_{-N}\left(z,w,t\right)\\
= & \frac{1}{t^{N+2}}\sum_{j=0}^{N}\binom{N+1}{j}\left(-1\right)^{N+2-j}q_{-N+j}\left(z,w,t\right)\\
 & \hspace{0.5cm}+\frac{1}{t^{N+2}}\sum_{j=1}^{N+1}\binom{N+1}{j-1}\left(-1\right)^{N+2-j}q_{-N+j}\left(z,w,t\right)-\frac{q_{1}\left(z,w,t\right)}{t^{N+2}}+\frac{q_{0}\left(w,z,t\right)}{t^{N+2}}\\
= & \frac{1}{t^{N+2}}\left[\sum_{j=0}^{N}\binom{N+2}{j}\left(-1\right)^{N+2-j}q_{-N+j}\left(z,w,t\right)-\left(N+2\right)q_{1}\left(z,w,t\right)+q_{0}\left(w,z,t\right)\right]\\
= & \frac{1}{t^{N+2}}\left[\sum_{j=0}^{N+1}\binom{N+2}{j}\left(-1\right)^{N+2-j}q_{-N+j}\left(z,w,t\right)+q_{0}\left(w,z,t\right)\right],
\end{align*}
which confirms that (\ref{eq:(dz)^k of q_(-N)}) is true for $k=N+2$.
Next, assume that (\ref{eq:(dz)^k of q_(-N)}) holds for some $k\geq N+2$.
Combining (\ref{eq:(dz)^k of q_nu expressed as sum of q_nu+j}), (\ref{eq:(dz)^k q_nu =00003D (dw)^k q_(nu+k)})
and (\ref{eq:(dz)q_1}), we have that $\partial_{z}^{k+1}q_{-N}\left(z,w,t\right)$
is equal to
\[
\begin{split} & \frac{1}{t^{k+1}}\sum_{j=0}^{N}\binom{k}{j}\left(-1\right)^{k+1-j}\left(q_{-N+j}\left(z,w,t\right)-q_{-N+j+1}\left(z,w,t\right)\right)+\frac{1}{t^{k+1}}\binom{k}{N+1}\left(-1\right)^{k-N}q_{1}\left(z,w,t\right)\\
 & \;-\frac{1}{t^{k+1}}\binom{k}{N+1}\left(-1\right)^{k-N}q_{0}\left(w,z,t\right)+\frac{1}{t^{k+1}}\sum_{j=N+2}^{k}\binom{k}{j}\left(-1\right)^{k+1-j}\left(q_{2+N-j}\left(w,z,t\right)-q_{1+N-j}\left(w,z,t\right)\right)\\
= & \frac{1}{t^{k+1}}\sum_{j=0}^{N}\binom{k+1}{j}\left(-1\right)^{k+1-j}q_{-N+j}\left(z,w,t\right)+\frac{1}{t^{k+1}}\binom{k+1}{N+1}\left(-1\right)^{k-N}q_{1}\left(z,w,t\right)\\
 & \hspace{0.5cm}+\frac{1}{t^{k+1}}\binom{k+1}{N+2}\left(-1\right)^{k-N-1}q_{0}\left(w,z,t\right)+\frac{1}{t^{k+1}}\sum_{j=N+3}^{k+1}\binom{k+1}{j}\left(-1\right)^{k+1-j}q_{2+N-j}\left(w,z,t\right)\\
= & \frac{1}{t^{k+1}}\sum_{j=0}^{N+1}\binom{k+1}{j}\left(-1\right)^{k+1-j}q_{-N+j}\left(z,w,t\right)+\frac{1}{t^{k+1}}\sum_{j=N+2}^{k+1}\binom{k+1}{j}\left(-1\right)^{k+1-j}q_{2+N-j}\left(w,z,t\right).
\end{split}
\]
Finally, by (\ref{eq:(dz)^k of q_(-N)}), we have that $\partial_{w}^{k}q_{N+2-k}\left(w,z,t\right)$,
with $k\geq N+2$, is equal to
\[
\begin{split} & \frac{1}{t^{k}}\sum_{j=0}^{k-N-1}\binom{k}{j}\left(-1\right)^{k-j}q_{2+N-k+j}\left(w,z,t\right)+\frac{1}{t^{k}}\sum_{j=k-N}^{k}\binom{k}{j}\left(-1\right)^{k-j}q_{k-N-j}\left(z,w,t\right)\\
= & \frac{1}{t^{k}}\sum_{l=N+1}^{k}\binom{k}{l}\left(-1\right)^{l}q_{2+N-l}\left(w,z,t\right)+\frac{1}{t^{k}}\sum_{l=0}^{N}\binom{k}{l}\left(-1\right)^{l}q_{-N+l}\left(z,w,t\right)\\
= & \left(-1\right)^{k}\partial_{z}^{k}q_{-N}\left(z,w,t\right)
\end{split}
\]
where in the last equality we used the fact that $q_{1}\left(w,z,t\right)=q_{1}\left(z,w,t\right)$
by (\ref{eq:symmetry of q_nu}).
\end{proof}
Next, we will look at the boundedness of the derivatives of $z\mapsto q_{\nu}\left(z,w,t\right)$
in a neighborhood of 0. For general $\nu$, the boundedness of $\partial_{z}^{k}q_{\nu}\left(z,w,t\right)$
for $z,w$ near 0 can be derived following (\ref{eq: estimate of q_nu}),
(\ref{eq:estimate of q_nu for nu>=00003D1}) and (\ref{eq:(dz)^k of q_nu expressed as sum of q_nu+j}).
But when $-\nu\in\mathbb{N}$, it is already clear from the series
representation in (\ref{eq:def of q_nu}) that, for every $t>0$,
$q_{\nu}\left(z,w,t\right)$ is analytic in $\left(z,w\right)$ on
$\left(0,\infty\right)^{2}$, which certainly implies the boundedness
of derivatives of all orders in any neighborhood of the origin. We
state these simple facts without proofs as follows.
\begin{cor}
\label{cor:regularity of bar q at 0} If $\nu\in\left(-\infty,1\right)\backslash\mathbb{Z}$,
then for every $t>0$, every $k\in\mathbb{N}$ and every $M>0$, 
\[
\sup_{\left(z,w\right)\in\left(0,M\right)^{2}}z^{\left(\nu+k-1\right)\vee0}\left|\partial_{z}^{k}q_{\nu}\left(z,w,t\right)\right|<\infty.
\]
In particular, $z\mapsto q_{\nu}\left(z,w,t\right)$ has bounded derivatives
up to the order of $\left[1-\nu\right]$ when $z,w$ are near 0.

If $-\nu\in\mathbb{N}$, then for every $t>0$, every $k\in\mathbb{N}$
and every $M>0$,
\[
\sup_{\left(z,w\right)\in(0,M)^{2}}\left|\partial_{z}^{k}q_{\nu}\left(z,w,t\right)\right|<\infty,
\]
i.e., $z\mapsto q_{\nu}\left(z,w,t\right)$ has bounded derivatives
of all orders when $z,w$ are near 0.
\end{cor}

Finally, we turn our attention to the derivatives of $v_{g}\left(z,t\right)$
in $z$. Seeing from Corollary \ref{cor:regularity of bar q at 0},
it is reasonable to split our discussions according to whether $\nu$
is a non-positive integer or not. 
\begin{prop}
\label{prop:derivative of v_g,nu, with nu not being integer}Assume
that $\nu\in\left(-\infty,1\right)\backslash\mathbb{Z}$. Given $g\in C_{b}\left(\left(0,\infty\right)\right)$,
let $v_{g}\left(z,t\right)$ be defined as in (\ref{eq: def of v_g(z,t)})
for the given value of $\nu$. If, for every $k\in\mathbb{N}$, we
set 
\begin{equation}
v_{g}^{\left(k\right)}\left(z,t\right):=\partial_{z}^{k}v_{g}\left(z,t\right)\text{ for }\left(z,t\right)\in\left(0,\infty\right)^{2},\label{eq:def of v^(k)_g}
\end{equation}
then $v_{g}^{\left(k\right)}\left(z,t\right)$ satisfies that, for
every $\left(z,t\right)\in\left(0,\infty\right)^{2}$,
\begin{equation}
\partial_{t}v_{g}^{\left(k\right)}\left(z,t\right)=z\partial_{z}^{2}v_{g}^{\left(k\right)}\left(z,t\right)+\left(\nu+k\right)\partial_{z}v_{g}^{\left(k\right)}\left(z,t\right)\label{eq:model equation of k derivative}
\end{equation}
and 
\begin{equation}
v_{g}^{\left(k\right)}\left(z,t\right)=\int_{0}^{\infty}\partial_{z}^{k}q_{\nu}\left(z,w,t\right)g\left(w\right)dw.\label{eq:expression of (dz)^k of v_g as integral of (dz)^k q_nu}
\end{equation}
In particular, for every $t>0$,
\[
\lim_{z\searrow0}v_{g}^{\left(k\right)}\left(z,t\right)=0\text{ if }k\in\left\{ 0,1,2,\cdots,\left[1-\nu\right]\right\} ,
\]
and
\[
\lim_{z\searrow0}z^{\nu-1+k}v_{g}^{\left(k\right)}\left(z,t\right)=\frac{\int_{0}^{\infty}e^{-\frac{w}{t}}g\left(w\right)dw}{\Gamma\left(2-\nu-k\right)t^{2-\nu}}\text{ if }k\ge\left[1-\nu\right]+1.
\]

Further, if $g\in C^{k}\left(\left(0,\infty\right)\right)$ is such
that $C_{k}^{g}<\infty$ and
\[
\lim_{z\searrow0}g^{\left(j\right)}\left(z\right)=0\text{ for every }j\in\left\{ 0,1,2,\cdots,k-1\right\} ,
\]
then 
\begin{equation}
v_{g}^{\left(k\right)}\left(z,t\right)=\int_{0}^{\infty}q_{\nu+k}\left(z,w,t\right)g^{\left(k\right)}\left(w\right)dw\text{ for every }\left(z,t\right)\in\left(0,\infty\right)^{2}\label{eq:expression of (dz)^k of v_g as integral of g^(k)}
\end{equation}
and in particular, 
\[
\lim_{t\searrow0}v_{g}^{\left(k\right)}\left(z,t\right)=g^{\left(k\right)}\left(z\right)\text{ for every }z>0.
\]
\end{prop}

\begin{proof}
Let $v_{g}\left(z,t\right)$ be defined as in (\ref{eq: def of v_g(z,t)}).
It follows from (\ref{eq: estimate of q_nu}) and (\ref{eq:estimate of q_nu for nu>=00003D1})
that one can compute the derivatives of $v_{g}\left(z,t\right)$ in
$z$ by differentiating under the integral sign (\ref{eq: def of v_g(z,t)}).
Combining with (\ref{eq:model equation}), we can easily see that
$v_{g}^{\left(k\right)}\left(z,t\right)$ is smooth on $\left(0,\infty\right)^{2}$
and satisfies (\ref{eq:model equation of k derivative}) and (\ref{eq:expression of (dz)^k of v_g as integral of (dz)^k q_nu})
for every $k\in\mathbb{N}$. By (\ref{eq:(dz)^k of q_nu expressed as sum of q_nu+j}),
we have that 
\[
v_{g}^{\left(k\right)}\left(z,t\right)=\frac{1}{t^{k}}\int_{0}^{\infty}\sum_{j=0}^{k}\binom{k}{j}\left(-1\right)^{k-j}q_{\nu+j}\left(z,w,t\right)g\left(w\right)dw.
\]
Thus, $v_{g}^{\left(k\right)}\left(z,t\right)$ has the stated limit
or asymptotics as $z\searrow0$ due to the simple fact that, for every
$\left(w,t\right)\in\left(0,\infty\right)^{2}$, as $z\searrow0$,
\[
\sum_{j=0}^{k}\binom{k}{j}\left(-1\right)^{k-j}q_{\nu+j}\left(z,w,t\right)\rightarrow0\text{ if }k\in\left\{ 0,1,2,\cdots,\left[1-\nu\right]\right\} 
\]
and 
\[
z^{\nu-1+k}\sum_{j=0}^{k}\binom{k}{j}\left(-1\right)^{k-j}q_{\nu+j}\left(z,w,t\right)\rightarrow\frac{e^{-\frac{w}{t}}}{\Gamma\left(2-\nu-k\right)t^{2-\nu-k}}\text{ if }k\geq\left[1-\nu\right]+1.
\]

As for the second statement of Proposition \ref{prop:derivative of v_g,nu, with nu not being integer},
given the extra hypothesis on $g$, we can easily derive (\ref{eq:expression of (dz)^k of v_g as integral of g^(k)})
from (\ref{eq:expression of (dz)^k of v_g as integral of (dz)^k q_nu})
by performing integration by parts multiple times. Given (\ref{eq:expression of (dz)^k of v_g as integral of g^(k)}),
to see that $v_{g}^{\left(k\right)}\left(z,t\right)$ has initial
data $g^{\left(k\right)}$ even when $\nu+k\geq1$, we simply apply
the same argument as the one used to show that $v_{g}\left(z,t\right)$
has initial data $g$. In particular, it suffices to notice that,
by (\ref{eq:estimate of q_nu for nu>=00003D1}), $\lim_{t\searrow0}\int_{0}^{\infty}q_{\nu+k}\left(z,w,t\right)dw=1$
and $\lim_{t\searrow0}\int_{(0,\infty)\backslash\left(z-\delta,z+\delta\right)}q_{\nu+k}\left(z,w,t\right)dw=0$
for every $\delta>0$, and the convergence in each limit is uniformly
fast in $z$ in any compact subset of $\left(0,\infty\right)$. 
\end{proof}
Now assume that $-\nu\in\mathbb{N}$, say, $\nu=-N$ for some $N\in\mathbb{N}$.
Since $z\mapsto q_{-N}\left(z,w,t\right)$ has bounded derivatives
of all orders near 0, we would expect that $v_{g}\left(z,t\right)$
has the same property. When $k\leq N+1$, it is easy to see that the
statements of Proposition \ref{prop:derivative of v_g,nu, with nu not being integer}
still apply to $v_{g}^{\left(k\right)}\left(z,t\right)$ with minor
changes in the expressions of initial data. However, when $k\geq N+2$,
(\ref{eq:(dz)^k of q_(-N)}) and (\ref{eq:(dz)^k q_(-N) =00003D (dw)^k q_(N+2-k)})
indicate that we should consider the operator $L_{N-k+2}^{*}$, the
adjoint of $L_{N-k+2}$. Indeed, $L_{-N+k}$ coincides with $L_{N-k+2}^{*}$,
and hence $v_{g}^{\left(k\right)}\left(z,t\right)$ is also a solution
to $\left(\partial_{t}-L_{N-k+2}^{*}\right)v_{g}^{\left(k\right)}\left(z,t\right)=0$.
We will make these considerations rigorous in the next proposition.
\begin{prop}
\label{prop:derivative of v_g,nu, with nu integer}Assume that $\nu=-N$
for some $N\in\mathbb{N}$. Given $g\in C_{b}\left(\left(0,\infty\right)\right)$,
let $v_{g}\left(z,t\right)$ be defined as in (\ref{eq: def of v_g(z,t)})
for the given value of $\nu$. For $k\in\mathbb{N}$, let $v_{g}^{\left(k\right)}\left(z,t\right)$
be defined as in (\ref{eq:def of v^(k)_g}). Then, for every $\left(z,t\right)\in\left(0,\infty\right)^{2}$,
$v_{g}^{\left(k\right)}\left(z,t\right)$ satisfies (\ref{eq:model equation of k derivative})
and (\ref{eq:expression of (dz)^k of v_g as integral of (dz)^k q_nu})
with 
\[
\lim_{z\searrow0}v_{g}^{\left(k\right)}\left(z,t\right)=\begin{cases}
0, & \text{ if }k\leq N,\\
\frac{1}{t^{N+2}}\int_{0}^{\infty}e^{-\frac{w}{t}}g\left(w\right)dw, & \text{ if }k=N+1,\\
\sum_{j=N+1}^{k}\binom{k}{j}\left(-1\right)^{k-j}\frac{\int_{0}^{\infty}e^{-\frac{w}{t}}w^{j-N-1}g\left(w\right)dw}{t^{k+j-N}\left(j-N-1\right)!}, & \text{ if }k\geq N+2.
\end{cases}
\]

Furthermore, if $g\in C^{k}\left(0,\infty\right)$ is such that $C_{k}^{g}<\infty$
and 
\[
\lim_{z\searrow0}g^{\left(j\right)}\left(z\right)=0\text{ for every }j\in\left\{ 0,1,\cdots,N\wedge\left(k-1\right)\right\} ,
\]
then for every $\left(z,t\right)\in\left(0,\infty\right)^{2}$,
\begin{equation}
v_{g}^{\left(k\right)}\left(z,t\right)=\begin{cases}
\int_{0}^{\infty}q_{-N+k}\left(z,w,t\right)g^{\left(k\right)}\left(w\right)dw, & \text{ if }k\leq N+1,\\
\int_{0}^{\infty}q_{2+N-k}\left(w,z,t\right)g^{\left(k\right)}\left(w\right)dw, & \text{ if }k\geq N+2,
\end{cases}\label{eq:expression of (dz)^k of v_g when nu=00003D-N}
\end{equation}
and
\[
\lim_{t\searrow0}v_{g}^{\left(k\right)}\left(z,t\right)=g^{\left(k\right)}\left(z\right)\text{ for every }z>0.
\]
\end{prop}

\begin{proof}
When $\nu=-N$, by (\ref{eq:(dz)^k of q_(-N)}), one can prove (\ref{eq:model equation of k derivative})
and (\ref{eq:expression of (dz)^k of v_g as integral of (dz)^k q_nu})
in exactly the same way as in Proposition \ref{prop:derivative of v_g,nu, with nu not being integer}.
Besides, it is clear from (\ref{eq:def of q_sigma}) and (\ref{eq:(dz)^k of q_(-N)})
that 
\[
\lim_{z\searrow0}\partial_{z}^{k}q_{-N}\left(z,w,t\right)=\begin{cases}
0, & \text{if }k\leq N,\\
\frac{1}{t^{2+N}}e^{-\frac{w}{t}}, & \text{if }k=N+1,\\
\sum_{j=N+1}^{k}\binom{k}{j}\left(-1\right)^{k-j}e^{-\frac{w}{t}}\frac{w^{j-N-1}}{t^{k+j-N}\left(j-N-1\right)!} & \text{if }k\geq N+2,
\end{cases}
\]
from which it follows that $v_{g}^{\left(k\right)}\left(z,t\right)$
has the stated boundary value. 

To prove (\ref{eq:expression of (dz)^k of v_g when nu=00003D-N}),
we notice that the proof of (\ref{eq:expression of (dz)^k of v_g as integral of g^(k)})
still applies to the case when $k\leq N+1$. In particular, it implies
that 
\[
v_{g}^{\left(N+1\right)}\left(z,t\right)=\int_{0}^{\infty}q_{1}\left(z,w,t\right)g^{\left(N+1\right)}\left(w\right)dw=\int_{0}^{\infty}q_{1}\left(w,z,t\right)g^{\left(N+1\right)}\left(w\right)dw,
\]
where again we used the symmetry of $q_{1}\left(z,w,t\right)$ in
$\left(z,w\right)$. In other words, (\ref{eq:expression of (dz)^k of v_g when nu=00003D-N})
is true for all $k\in\left\{ 0,1,2,\cdots,N+1\right\} $. Assume that
(\ref{eq:expression of (dz)^k of v_g when nu=00003D-N}) holds for
some $k\geq N+1$. Following (\ref{eq:(dz)^k q_(-N) =00003D (dw)^k q_(N+2-k)})
and the hypothesis on $g$, we have that 
\[
\begin{split}v_{g}^{\left(k+1\right)}\left(z,t\right) & =\int_{0}^{\infty}\partial_{z}q_{2+N-k}\left(w,z,t\right)g^{\left(k\right)}\left(w\right)dw\\
 & =-\int_{0}^{\infty}\partial_{w}q_{1+N-k}\left(w,z,t\right)g^{\left(k\right)}\left(w\right)dw\\
 & =\int_{0}^{\infty}q_{2+N-\left(k+1\right)}\left(w,z,t\right)g^{\left(k+1\right)}\left(w\right)dw
\end{split}
\]
which validates (\ref{eq:expression of (dz)^k of v_g when nu=00003D-N})
for $k+1$.
\end{proof}
\begin{rem}
We will finish this section with two remarks on the derivatives of
$q_{\nu}\left(z,w,t\right)$. The first remark is that, when $\nu=-N$
and $k\geq N+2$, we observe that 
\[
q_{N+2-k}\left(w,z,t\right)=q_{-N+k}^{*}\left(z,w,t\right)\text{ for every }\left(z,w,t\right)\in\left(0,\infty\right)^{3},
\]
where $q_{-N+k}^{*}\left(z,w,t\right)$ is the fundamental solution
to the equation in (\ref{eq:model equation}) under the zero flux
boundary condition, as defined in Remark \ref{rem:the other formula for q_nu}.
This is not surprising, because, since it has bounded derivatives
of all orders near 0, $v_{g}^{\left(k\right)}\left(z,t\right)$ is
a solution to (\ref{eq:model equation of k derivative}) that satisfies
the zero flux boundary condition. 

The second remark is on a simplification of the notations involving
``$q_{\nu+k}\left(z,w,t\right)$''. Namely, for our purpose of studying
$\partial_{z}^{k}q_{\nu}\left(z,w,t\right)$ and $v_{g}^{\left(k\right)}\left(z,t\right)$
as in Lemma \ref{lem:technical lemma} and Proposition \ref{prop:derivative of v_g,nu, with nu integer},
$q_{2+N-k}\left(w,z,t\right)$ for $k\geq N+2$ plays the same role
as $q_{-N+k}\left(z,w,t\right)$ for $k\leq N+1$. Therefore, for
the convenience of notations, we further extend the definition of
$q_{\nu+k}\left(z,w,t\right)$ by setting, for every $\left(z,w,t\right)\in\left(0,\infty\right)^{3}$,
\begin{equation}
Q_{\nu+k}\left(z,w,t\right):=\begin{cases}
q_{\nu+k}\left(z,w,t\right), & \text{when }\nu\in\left(-\infty,1\right)\backslash\mathbb{Z},k\in\mathbb{N},\\
q_{2-\nu-k}\left(w,z,t\right), & \text{when }-\nu\in\mathbb{N},k\geq2-\nu.
\end{cases}\label{eq:def of Q_nu}
\end{equation}
Under this new notation, it is easy to see that, for any $\nu<1$
and $k,l\in\mathbb{N}$, (\ref{eq:(dz)^k of q_nu expressed as sum of q_nu+j}),
(\ref{eq:(dz)^k q_nu =00003D (dw)^k q_(nu+k)}), (\ref{eq:(dz)^k of q_(-N)})
and (\ref{eq:(dz)^k q_(-N) =00003D (dw)^k q_(N+2-k)}) can be combined
into the following relation:
\begin{equation}
\partial_{z}^{k}Q_{\nu+l}\left(z,w,t\right)=\left(-1\right)^{k}\partial_{w}^{k}Q_{\nu+l+k}\left(z,w,t\right)=\frac{1}{t^{k}}\sum_{j=0}^{k}\binom{k}{j}\left(-1\right)^{k-j}Q_{\nu+l+j}\left(z,w,t\right).\label{eq:(dz)^k of Q_nu+l}
\end{equation}
Similarly, (\ref{eq:expression of (dz)^k of v_g as integral of g^(k)})
and (\ref{eq:expression of (dz)^k of v_g when nu=00003D-N}) also
merge into one statement that, for every $k\in\mathbb{N}$, if $g\in C^{k}\left(\left(0,\infty\right)\right)$
is such that $C_{k}^{g}<\infty$ and 
\[
\lim_{z\searrow0}g^{\left(j\right)}\left(z\right)=0\text{ for }\begin{cases}
j\in\left\{ 0,1,\cdots,k-1\right\}  & \text{ when }\nu\in\left(-\infty,1\right)\backslash\mathbb{Z},\\
j\in\left\{ 0,1,\cdots,\left(k-1\right)\wedge\left(-\nu\right)\right\}  & \text{ when }-\nu\in\mathbb{N},
\end{cases}
\]
then
\begin{equation}
v_{g}^{\left(k\right)}\left(z,t\right)=\int_{0}^{\infty}Q_{\nu+k}\left(z,w,t\right)g^{\left(k\right)}\left(w\right)dw\text{ for every }\left(z,t\right)\in\left(0,\infty\right)^{2}.\label{eq:expression of (dz)^k of v_g as integral of g^(k) using Q}
\end{equation}
\end{rem}

\section{General Equation}

In this section we will take several steps to construct the fundamental
solution $p\left(x,y,t\right)$ to the general problem (\ref{eq:general IVP equation})
based on $q_{\nu}\left(z,w,t\right)$ and the perturbation techniques.
Throughout this section we will assume that $a\left(x\right)$ and
$b\left(x\right)$ satisfy Condition 1-3 as proposed in $\mathsection1.2$,
and hence we always have $\nu<1$. 

\subsection{The model equation with an extra drift.}

To connect (\ref{eq:general IVP equation}) and (\ref{eq:model equation}),
we begin with a change of variables that turns (\ref{eq:general IVP equation})
into a variation of (\ref{eq:model equation}) with an extra drift.
Recall that for $x>0$,
\[
\phi\left(x\right):=\frac{1}{4}\left(\int_{0}^{x}\frac{ds}{\sqrt{a\left(s\right)}}\right)^{2}\text{ and }d\left(x\right):=\frac{1}{2}+\frac{2b\left(x\right)-a^{\prime}\left(x\right)}{2\sqrt{a\left(x\right)}}\sqrt{\phi\left(x\right)}-\nu
\]
where
\[
\nu=\frac{1}{2}+\lim_{x\searrow0}\frac{2b\left(x\right)-a^{\prime}\left(x\right)}{2\sqrt{a\left(x\right)}}\sqrt{\phi\left(x\right)}<1.
\]
Besides, $\psi:z\in\left(0,\infty\right)\mapsto\psi\left(z\right)\in\left(0,\infty\right)$
is the inverse function of $\phi$ and $\tilde{d}:=d\circ\psi$. Consider
the following Cauchy initial value problem with the Dirichlet boundary
condition:
\begin{equation}
\begin{array}{c}
\partial_{t}\tilde{v}_{g}\left(z,t\right)=z\partial_{z}^{2}\tilde{v}_{g}\left(z,t\right)+\left(\nu+\tilde{d}\left(z\right)\right)\partial_{z}\tilde{v}_{g}\left(z,t\right)\text{ for }\left(z,t\right)\in\left(0,\infty\right)^{2},\\
\lim_{t\searrow0}\tilde{v}_{g}\left(z,t\right)=g\left(z\right)\text{ for }z\in\left(0,\infty\right)\text{ and }\lim_{z\searrow0}\tilde{v}_{g}\left(z,t\right)=0\text{ for }t\in\left(0,\infty\right).
\end{array}\label{eq:model eq with drift}
\end{equation}
\begin{lem}
\label{lem:change of variable} Give $f\in C_{b}\left(\left(0,\infty\right)\right)$
and $g:=f\circ\psi$, $\tilde{v}_{g}\left(z,t\right)\in C^{2,1}\left(\left(0,\infty\right)^{2}\right)$
is a solution to (\ref{eq:model eq with drift}) with initial data
$g$ if and only if
\begin{equation}
u_{f}\left(x,t\right):=\tilde{v}_{g}\left(\phi\left(x\right),t\right)\in C^{2,1}\left(\left(0,\infty\right)^{2}\right),\label{eq:u_f defined through v^tilde_g}
\end{equation}
is a solution to the original problem (\ref{eq:general IVP equation})
with initial data $f$. 
\end{lem}

We will omit the proof since everything can be verified by direct
computations.

Given Lemma \ref{lem:change of variable}, our plan becomes clear
that, in order to solve (\ref{eq:general IVP equation}), we will
transform it to (\ref{eq:model eq with drift}) where the diffusion
coefficient degenerates linearly at 0 and the drift $\nu+\tilde{d}\left(z\right)$
is smooth on $\left(0,\infty\right)$ and is approximately $\nu$
near 0 since 
\[
\lim_{z\searrow0}\tilde{d}\left(z\right)=\lim_{x\searrow0}d\left(x\right)=0.
\]
Another advantage of (\ref{eq:model eq with drift}) is that, according
to (\ref{eq:conditions on b^tilde and z}), $\nu+\tilde{d}\left(z\right)$
is Lipschitz continuous on $\left(0,\infty\right)$, and hence the
Yamada-Watanabe theorem guarantees the existence of the almost surely
unique solution $\left\{ \tilde{Y}\left(z,t\right):\left(z,t\right)\in\left[0,\infty\right)^{2}\right\} $
to the equation
\begin{equation}
\tilde{Y}\left(z,t\right):=z+\int_{0}^{t}\sqrt{2\left|\tilde{Y}\left(z,s\right)\right|}dB\left(s\right)+\nu t+\int_{0}^{t}\tilde{d}\left(\tilde{Y}\left(z,s\right)\right)ds\text{ for }\left(z,t\right)\in\left[0,\infty\right)^{2}\label{eq:SDE satisfied by Y^tilde}
\end{equation}
with the constraint that $\tilde{Y}\left(z,t\right)\equiv0$ for every
$z\ge0$ and $t\geq\zeta_{0}^{\tilde{Y}}\left(z\right)$. Meanwhile,
if we define 
\begin{equation}
X\left(x,t\right):=\psi\left(\tilde{Y}\left(\phi\left(x\right),t\right)\right)\text{ for }\left(x,t\right)\in\left[0,\infty\right)^{2},\label{eq:transformation between X and Y}
\end{equation}
then one can follow Itô's formula to check that
\begin{align*}
X\left(x,t\right) & =x+\int_{0}^{t}\psi^{\prime}\left(\tilde{Y}\left(\phi\left(x\right),s\right)\right)\sqrt{2\left|\tilde{Y}\left(\phi\left(x\right),s\right)\right|}dB\left(s\right)\\
 & \hspace{1cm}+\int_{0}^{t}\left[\tilde{Y}\left(\phi\left(x\right),s\right)\psi^{\prime\prime}\left(\tilde{Y}\left(\phi\left(x\right),s\right)\right)+\psi^{\prime}\left(\tilde{Y}\left(\phi\left(x\right),s\right)\right)\left(\nu+\tilde{d}\left(\tilde{Y}\left(\phi\left(x\right),s\right)\right)\right)\right]ds\\
 & =x+\int_{0}^{t}\sqrt{2a\left(\psi\left(\tilde{Y}\left(\phi\left(x\right),s\right)\right)\right)}dB\left(s\right)+\int_{0}^{t}b\left(\psi\left(\tilde{Y}\left(\phi\left(x\right),s\right)\right)\right)ds\\
 & =x+\int_{0}^{t}\sqrt{2a\left(X\left(x,s\right)\right)}dB\left(s\right)+\int_{0}^{t}b\left(X\left(x,s\right)\right)ds.
\end{align*}
In other words, although the Yamada-Watanabe theorem does not apply
directly to the equation with $a\left(x\right)$ and $b\left(x\right)$,
we have managed to find a process $\left\{ X\left(x,t\right):\left(x,t\right)\in\left[0,\infty\right)^{2}\right\} $
that satisfies (\ref{eq: Ito integral equation for (i)}) and (\ref{eq:boundary constraint on X(x,t)}).
Since $\left\{ \tilde{Y}\left(z,t\right):\left(z,t\right)\in\left[0,\infty\right)^{2}\right\} $
is the almost surely unique solution to (\ref{eq:SDE satisfied by Y^tilde})
and $\phi:[0,\infty)\rightarrow[0,\infty)$ is a diffeomorphism, $\left\{ X\left(x,t\right):\left(x,t\right)\in\left[0,\infty\right)^{2}\right\} $
is also the almost surely unique solution to (\ref{eq: Ito integral equation for (i)})
and hence has the strong Markov property. 

We summarize the findings above in the following proposition. We will
omit the proof since it is exactly the same as that of Proposition
\ref{prop:uniqueness of v_g  and CK eq for q_nu}.
\begin{prop}
\label{prop:uniqueness of bar v_g  and u_f}Given $f\in C_{b}\left(\left(0,\infty\right)\right)$
and $g:=f\circ\psi$, if $\tilde{v}_{g}\left(z,t\right)\in C^{2,1}\left(\left(0,\infty\right)^{2}\right)$
is a solution to (\ref{eq:model eq with drift}) with initial data
$g$, then
\[
\begin{split}\tilde{v}_{g}\left(z,t\right) & =\mathbb{E}\left[g\left(\tilde{Y}\left(z,t\right)\right);t<\zeta_{0}^{\tilde{Y}}\left(z\right)\right]\end{split}
\text{ for every }\left(z,t\right)\in\left(0,\infty\right)^{2},
\]
and hence $\tilde{v}_{g}\left(z,t\right)$ is the unique solution
in $C^{2,1}\left(\left(0,\infty\right)^{2}\right)$ to (\ref{eq:model eq with drift}). 

Further, if $u_{f}\left(x,t\right)$ is defined as in (\ref{eq:u_f defined through v^tilde_g}),
then
\[
u_{f}\left(x,t\right)=\mathbb{E}\left[f\left(X\left(x,t\right)\right);t<\zeta_{0}^{X}\left(x\right)\right]\text{ for every }\left(x,t\right)\in\left(0,\infty\right)^{2},
\]
and hence $u_{f}\left(x,t\right)$ is the unique solution in $C^{2,1}\left(\left(0,\infty\right)^{2}\right)$
to (\ref{eq:general IVP equation}). 
\end{prop}

It should be clear that to proceed from here, we will treat (\ref{eq:model eq with drift})
as a perturbation of the model equation (\ref{eq:model equation}),
and study the fundamental solution to (\ref{eq:model eq with drift})
based on the results we have established on $q_{\nu}\left(z,w,t\right)$.
To achieve this goal, we will need another variation of (\ref{eq:model equation}).

\subsection{The model equation with a potential.}

To connect (\ref{eq:model eq with drift}) and (\ref{eq:model equation}),
we will first turn the extra drift $\tilde{d}\left(z\right)$ into
a potential, and seek to invoke the Duhamel perturbation method. To
this end, we define 
\[
\theta:\;z\in\left(0,\infty\right)\mapsto\theta\left(z\right):=\exp\left(-\int_{0}^{z}\frac{\tilde{d}(u)}{2u}du\right)\in\left(0,\infty\right),
\]
and $\theta$ is positive and smooth on $\left(0,\infty\right)$.
Further, if we define 
\[
V:\;z\in\left(0,\infty\right)\mapsto V\left(z\right):=-\frac{\tilde{d}^{2}\left(z\right)}{4z}-\frac{\tilde{d}^{\prime}\left(z\right)}{2}+\left(1-\nu\right)\frac{\tilde{d}\left(z\right)}{2z},
\]
then $V\left(z\right)$ is smooth and uniformly bounded on $\left(0,\infty\right)$,
according to (\ref{eq:conditions on b^tilde and z}). 
\begin{lem}
\label{lem:drift to potential transformation}Given $g\in C_{b}\left(\left(0,\infty\right)\right)$
and $h:=\frac{g}{\theta}$, $\tilde{v}_{g}\left(z,t\right)\in C^{2,1}\left(\left(0,\infty\right)^{2}\right)$
is a solution to (\ref{eq:model eq with drift}) with initial data
$g$ if and only if 
\[
v_{h}^{V}\left(z,t\right):=\frac{\tilde{v}_{g}\left(z,t\right)}{\theta\left(z\right)}\in C^{2,1}\left(\left(0,\infty\right)^{2}\right)
\]
is a solution to 
\begin{equation}
\begin{array}{c}
\begin{array}{c}
\partial_{t}v_{h}^{V}\left(z,t\right)=z\partial_{z}^{2}v_{h}^{V}\left(z,t\right)+\nu\partial_{z}v_{h}^{V}\left(z,t\right)+V\left(z\right)v_{h}^{V}\left(z,t\right)\text{ for }\left(z,t\right)\in\left(0,\infty\right)^{2},\\
\lim_{t\searrow0}v_{h}^{V}\left(z,t\right)=h\left(z\right)\text{ for }z\in\left(0,\infty\right)\text{ and }\lim_{z\searrow0}v_{h}^{V}\left(z,t\right)=0\text{ for }t\in\left(0,\infty\right).
\end{array}\end{array}\label{eq:model equation with potential}
\end{equation}
\end{lem}

Again, we will omit the proof to Lemma \ref{lem:drift to potential transformation}
since it is straightforward.

Following the method of Duhamel, in order to solve (\ref{eq:model equation with potential}),
we need to find a function $q_{\nu}^{V}\left(z,w,t\right)$ that solves
the integral equation 
\begin{equation}
q_{\nu}^{V}\left(z,w,t\right)=q_{\nu}\left(z,w,t\right)+\int_{0}^{t}\int_{0}^{\infty}q_{\nu}\left(z,\xi,t-\tau\right)q_{\nu}^{V}\left(\xi,w,\tau\right)V\left(\xi\right)d\xi d\tau.\label{eq: duhamel integral eq}
\end{equation}
To this end, for every $\left(z,w,t\right)\in\left(0,\infty\right)^{3}$,
we set $q_{\nu,0}\left(z,w,t\right):=q_{\nu}\left(z,w,t\right)$ and
recursively define
\begin{equation}
q_{\nu,n+1}\left(z,w,t\right):=\int_{0}^{t}\int_{0}^{\infty}q_{\nu}\left(z,\xi,t-\tau\right)q_{\nu,n}\left(\xi,w,\tau\right)V\left(\xi\right)d\xi d\tau\text{ for }n\geq0.\label{eq:recursion n->n+1}
\end{equation}
\begin{lem}
\label{lem:def of q^V_=00005Cnu} For every $\left(z,w,t\right)\in\left(0,\infty\right)^{3}$,
\begin{equation}
q_{\nu}^{V}\left(z,w,t\right):=\sum_{n=0}^{\infty}q_{\nu,n}\left(z,w,t\right)\label{eq:def of q^V_nu}
\end{equation}
is well defined as an absolutely convergent series,
\begin{equation}
\left|q_{\nu}^{V}\left(z,w,t\right)\right|\leq e^{t\left\Vert V\right\Vert _{u}}q_{\nu}\left(z,w,t\right)\label{eq:exp estimate for q^V_=00005Cnu}
\end{equation}
and 
\begin{equation}
\sup_{z,w\in\left(0,\infty\right)^{2}}\left|\frac{q_{\nu}^{V}\left(z,w,t\right)}{q_{\nu}\left(z,w,t\right)}-1\right|\leq e^{t\left\Vert V\right\Vert _{u}}-1.\label{eq: estimate of q^V_nu/q_nu}
\end{equation}
Moreover, $q_{\nu}^{V}\left(z,w,t\right)$ satisfies the integral
equation (\ref{eq: duhamel integral eq}). 
\end{lem}

\begin{proof}
By (\ref{eq:CK equation for q_nu}), (\ref{eq:recursion n->n+1})
and a simple application of induction, we can check that for every
$n\in\mathbb{N}$,
\begin{equation}
\left|q_{\nu,n}\left(z,w,t\right)\right|\leq\frac{\left(t\left\Vert V\right\Vert _{u}\right)^{n}}{n!}q_{\nu}\left(z,w,t\right)\text{ for every }\left(z,w,t\right)\in\left(0,\infty\right)^{3}.\label{eq:estimate for q_nu,n}
\end{equation}
Therefore, the series $q_{\nu}^{V}\left(z,w,t\right):=\sum_{n=0}^{\infty}q_{\nu,n}\left(z,w,t\right)$
is absolutely convergent and

\[
\sum_{n=0}^{\infty}\left|q_{\nu,n}\left(z,w,t\right)\right|\leq e^{t\left\Vert V\right\Vert _{u}}q_{\nu}\left(z,w,t\right),
\]
which gives (\ref{eq:exp estimate for q^V_=00005Cnu}) and (\ref{eq: estimate of q^V_nu/q_nu}).
(\ref{eq:estimate for q_nu,n}) also guarantees that one can plug
the series in (\ref{eq:def of q^V_nu}) into both sides of (\ref{eq: duhamel integral eq})
to verify its validity.
\end{proof}
Certainly the estimate (\ref{eq: estimate of q^V_nu/q_nu}) is more
meaningful when $t$ is small, in which case the effect of the potential
$V\left(z\right)$ has not become substantial and we do expect that
$q_{\nu}^{V}\left(z,w,t\right)$ is close to $q_{\nu}\left(z,w,t\right)$. 
\begin{prop}
\label{prop:results on model eq with potential}Given a function $h:\left(0,\infty\right)\rightarrow\mathbb{R}$
such that $h\cdot\theta\in C_{b}\left(\left(0,\infty\right)\right)$,
if we define
\begin{equation}
v_{h}^{V}\left(z,t\right):=\int_{0}^{\infty}q_{\nu}^{V}\left(z,w,t\right)h\left(w\right)dw\text{ for }\left(z,t\right)\in\left(0,\infty\right)^{2},\label{eq:def of v^V_h}
\end{equation}
then $v_{h}^{V}\left(z,t\right)$ is a smooth solution to (\ref{eq:model equation with potential}). 

Recall that $\left\{ Y\left(z,t\right):\left(z,t\right)\in[0,\infty)^{2}\right\} $
is the unique solution to (\ref{eq:SDE satisfied by Y(z,t)}). Then,
\begin{equation}
v_{h}^{V}\left(z,t\right)=\mathbb{E}\left[e^{\int_{0}^{t}V\left(Y\left(z,\tau\right)\right)d\tau}h\left(Y\left(z,t\right)\right);t<\zeta_{0}^{Y}\left(z\right)\right]\text{ for every }\left(z,t\right)\in\left(0,\infty\right)^{2},\label{eq:prob interpretation of q^V_nu}
\end{equation}
and hence $v_{h}^{V}\left(z,t\right)$ is the unique solution in $C^{2,1}\left(\left(0,\infty\right)^{2}\right)$
to (\ref{eq:model equation with potential}). 
\end{prop}

\begin{proof}
Let $h$ be a continuous function such that $h\cdot\theta$ is bounded
on $\left(0,\infty\right)$. Then, according to Condition 3, there
exist $C,C^{\prime}>0$ such that
\begin{equation}
\left|h\left(z\right)\right|\leq C^{\prime}e^{C\sqrt{z}}\text{ for every }z>0.\label{eq:estimate on initial h}
\end{equation}
Following exactly the same proof as that of Proposition \ref{prop: q_nu is the fundamental solution},
we can get that $v_{h}\left(z,t\right):=\int_{0}^{\infty}q_{\nu}\left(z,w,t\right)h\left(w\right)dw$
is a smooth solution to (\ref{eq:model equation}) with initial data
$h$. Then, (\ref{eq: duhamel integral eq}) implies that $v_{h}^{V}\left(z,t\right)$
and $v_{h}\left(z,t\right)$ have the relation that, for every $\left(z,t\right)\in\left(0,\infty\right)^{2}$,
\begin{equation}
\begin{split}v_{h}^{V}\left(z,t\right) & =v_{h}\left(z,t\right)+\end{split}
\int_{0}^{t}\int_{0}^{\infty}q_{\nu}\left(z,\xi,t-\tau\right)v_{h}^{V}\left(\xi,\tau\right)V\left(\xi\right)d\xi d\tau.\label{eq:relation between v_h  and v_h^V}
\end{equation}
It is easy to see from (\ref{eq:def of v^V_h}) that $\lim_{z\searrow0}v_{h}^{V}\left(z,t\right)=0$
for every $t>0$. We also observe that, by (\ref{eq: estimate of q^V_nu/q_nu}),
\[
\begin{split}\left|q_{\nu}^{V}\left(z,w,t\right)-q_{\nu}\left(z,w,t\right)\right| & \leq\left(e^{t\left\Vert V\right\Vert _{u}}-1\right)q_{\nu}\left(z,w,t\right)\text{ for every }\end{split}
\left(z,w,t\right)\in\left(0,\infty\right)^{3}.
\]
Thus, (\ref{eq: estimate of q_nu}) and (\ref{eq:estimate on initial h})
imply that
\[
\begin{split}\left|v_{h}^{V}\left(z,t\right)-v_{h}\left(z,t\right)\right| & \leq C^{\prime}\left(e^{t\left\Vert V\right\Vert _{u}}-1\right)\int_{0}^{\infty}q_{\nu}\left(z,w,t\right)e^{C\sqrt{w}}dw\\
 & \rightarrow0\text{ as }t\searrow0\text{ for every }z>0,
\end{split}
\]
which leads to 
\[
\lim_{t\searrow0}v_{h}^{V}\left(z,t\right)=\lim_{t\searrow0}v_{h}\left(z,t\right)=h\left(z\right)\text{ for every }z>0.
\]

Therefore, to prove the first statement of Proposition \ref{prop:results on model eq with potential},
the only thing left to do is to show that $v_{h}^{V}\left(z,t\right)$
is a smooth solution to the equation in (\ref{eq:model equation with potential}),
for which we will apply the hypoellipticity theory again. Given a
Schwartz function $\varphi$ on $\left(0,\infty\right)$, we consider
\[
\left\langle \varphi,v_{h}^{V}\left(\cdot,t\right)\right\rangle :=\int_{0}^{\infty}v_{h}^{V}\left(z,t\right)\varphi\left(z\right)dz\text{ for }t>0
\]
and use (\ref{eq:relation between v_h  and v_h^V}) to write it as
\[
\left\langle \varphi,v_{h}^{V}\left(\cdot,t\right)\right\rangle =\left\langle \varphi,v_{h}\left(\cdot,t\right)\right\rangle +\int_{0}^{t}\int_{0}^{\infty}\left\langle \varphi,q_{\nu}\left(\cdot,\xi,t-\tau\right)\right\rangle v_{h}^{V}\left(\xi,\tau\right)V\left(\xi\right)d\xi d\tau.
\]
Taking the derivative in $t$ of the equation above results in  
\[
\begin{split}\frac{d}{dt}\left\langle \varphi,v_{h}^{V}\left(\cdot,t\right)\right\rangle  & =\frac{d}{dt}\left\langle \varphi,v_{h}\left(\cdot,t\right)\right\rangle +\left\langle V\varphi,v_{h}^{V}\left(\cdot,t\right)\right\rangle \\
 & \qquad+\int_{0}^{t}\int_{0}^{\infty}\partial_{t}\left\langle \varphi,q_{\nu}\left(\cdot,\xi,t-\tau\right)\right\rangle v_{h}^{V}\left(\xi,\tau\right)V\left(\xi\right)d\xi d\tau\\
 & =\left\langle L_{\nu}^{*}\varphi,v_{h}\left(\cdot,t\right)\right\rangle +\left\langle V\varphi,v_{h}^{V}\left(\cdot,t\right)\right\rangle \\
 & \qquad+\int_{0}^{t}\int_{0}^{\infty}\left\langle L_{\nu}^{*}\varphi,q_{\nu}\left(\cdot,\xi,t-\tau\right)\right\rangle v_{h}^{V}\left(\xi,\tau\right)V\left(\xi\right)d\xi d\tau\\
 & =\left\langle \left(L_{\nu}^{*}+V\right)\varphi,v_{h}^{V}\left(\cdot,t\right)\right\rangle ,
\end{split}
\]
Therefore, $v_{h}^{V}\left(z,t\right)$ solves (\ref{eq:model equation with potential})
in the sense of tempered distributions. Since $\partial_{t}-z\partial_{z}^{2}-\nu\partial_{z}-V$
on $\left(0,\infty\right)$ is hypoelliptic ($\mathsection7.4$ of
\cite{PDEStroock}), we get that $v_{h}^{V}\left(z,t\right)$ is a
smooth solution to (\ref{eq:model equation with potential}). 

Next we get down to proving (\ref{eq:prob interpretation of q^V_nu}),
which is very similar to the proof of (\ref{eq:prob interpretation of q_nu}).
For every $\left(z,t\right)\in\left(0,\infty\right)^{2}$, by Itô's
formula and Doob's stopping time theorem 
\[
\left\{ e^{\int_{0}^{s\wedge\zeta_{0}^{Y}\left(z\right)}V\left(Y\left(z,\tau\right)\right)d\tau}v_{h}^{V}\left(Y\left(z,s\wedge\zeta_{0}^{Y}\left(z\right)\right),t-s\wedge\zeta_{0}^{Y}\left(z\right)\right):s\in\left[0,t\right]\right\} 
\]
is a martingale. Equating the expectation of the martingale at $0$
and $t$ leads to 
\[
\begin{split}v_{h}^{V}\left(z,t\right) & =\mathbb{E}\left[e^{\int_{0}^{t\wedge\zeta_{0}^{Y}\left(z\right)}V\left(Y\left(z,\tau\right)\right)d\tau}v_{h}^{V}\left(Y\left(z,t\wedge\zeta_{0}^{Y}\left(z\right)\right),t-t\wedge\zeta_{0}^{Y}\left(z\right)\right)\right]\\
 & =\mathbb{E}\left[e^{\int_{0}^{t}V\left(Y\left(z,\tau\right)\right)d\tau}h\left(Y\left(z,t\right)\right);t<\zeta_{0}^{Y}\left(z\right)\right].
\end{split}
\]
\end{proof}
We summarize the properties of $q_{\nu}^{V}\left(z,w,t\right)$ in
the next proposition.
\begin{prop}
\label{prop:results on q^V_nu}Let $q_{\nu}^{V}\left(z,w,t\right)$
be defined as in (\ref{eq:def of q^V_nu}). Then, for every $\left(z,w,t\right)\in\left(0,\infty\right)^{3}$,

\begin{equation}
w^{1-\nu}q_{\nu}^{V}\left(z,w,t\right)=z^{1-\nu}q_{\nu}^{V}\left(w,z,t\right),\label{eq:symmetry of q^V_nu}
\end{equation}
and $q_{\nu}^{V}\left(z,w,t\right)$ also satisfies the following
integral equation:
\begin{equation}
q_{\nu}^{V}\left(z,w,t\right)=q_{\nu}\left(z,w,t\right)+\int_{0}^{t}\int_{0}^{\infty}q_{\nu}^{V}\left(z,\xi,t-s\right)q_{\nu}\left(\xi,w,s\right)V\left(\xi\right)d\xi ds.\label{eq: duhamel integral equiv}
\end{equation}

Besides, for every $w>0$, $\left(z,t\right)\mapsto q_{\nu}^{V}\left(z,w,t\right)$
is a smooth solution to the equation in (\ref{eq:model equation with potential}),
and for every $z>0$, $\left(w,t\right)\mapsto q_{\nu}^{V}\left(z,w,t\right)$
is a smooth solution to the corresponding Kolmogorov forward equation.
Moreover, $q_{\nu}^{V}\left(z,w,t\right)$ is the fundamental solution
to (\ref{eq:model equation with potential}). 

Finally, $q_{\nu}^{V}\left(z,w,t\right)$ satisfies the Chapman-Kolmogorov
equation, i.e., for $z,w>0$ and $t,s>0$,
\begin{equation}
q_{\nu}^{V}\left(z,w,t+s\right)=\int_{0}^{\infty}q_{\nu}^{V}\left(z,\xi,t\right)q_{\nu}^{V}\left(\xi,w,s\right)d\xi.\label{eq:CK for q^V_nu}
\end{equation}
\end{prop}

\begin{proof}
To prove (\ref{eq:symmetry of q^V_nu}), we first note that if we
define $\tilde{q}_{\nu,0}\left(z,w,t\right):=q_{\nu}\left(z,w,t\right)$
and for every $n\geq0$, 
\begin{equation}
\tilde{q}_{\nu,n+1}\left(z,w,t\right):=\int_{0}^{t}\int_{0}^{\infty}\tilde{q}_{\nu,n}\left(z,\xi,t-\tau\right)q_{\nu}\left(\xi,w,\tau\right)V\left(\xi\right)d\xi d\tau,\label{eq:alternative of n->n+1}
\end{equation}
then $\tilde{q}_{\nu,n}\left(z,w,t\right)=q_{\nu,n}\left(z,w,t\right)$
for every $n\in\mathbb{N}$. In other words, (\ref{eq:alternative of n->n+1})
is an equivalent recursive relation to (\ref{eq:recursion n->n+1}).
To see this, one can expand both the right hand side of (\ref{eq:recursion n->n+1})
and that of (\ref{eq:alternative of n->n+1}) into two respective
$2n-$fold integrals, and observe that the two integrals are identical.
Next, we will show by induction that for every $n\geq0$,
\[
w^{1-\nu}q_{\nu,n}\left(z,w,t\right)=z^{1-\nu}q_{\nu,n}\left(w,z,t\right).
\]
When $n=0$, the relation is just (\ref{eq:symmetry of q_nu}). Assume
the relation holds for some $n\in\mathbb{N}$, by the equivalence
between (\ref{eq:recursion n->n+1}) and (\ref{eq:alternative of n->n+1}),
we have that
\[
\begin{split}w^{1-\nu}q_{\nu,n+1}\left(z,w,t\right) & =\int_{0}^{t}\int_{0}^{\infty}q_{\nu}\left(z,\xi,t-\tau\right)w^{1-\nu}q_{\nu,n}\left(\xi,w,\tau\right)V\left(\xi\right)d\xi d\tau\\
 & =\int_{0}^{t}\int_{0}^{\infty}q_{\nu}\left(z,\xi,t-\tau\right)\xi^{1-\nu}q_{\nu,n}\left(w,\xi,\tau\right)V\left(\xi\right)d\xi d\tau\\
 & =z^{1-\nu}\int_{0}^{t}\int_{0}^{\infty}q_{\nu}\left(\xi,z,t-\tau\right)q_{\nu,n}\left(w,\xi,\tau\right)V\left(\xi\right)d\xi d\tau\\
 & =z^{1-\nu}\int_{0}^{t}\int_{0}^{\infty}\tilde{q}_{\nu,n}\left(w,\xi,\tau\right)q_{\nu}\left(\xi,z,t-\tau\right)V\left(\xi\right)d\xi d\tau\\
 & =z^{1-\nu}\tilde{q}_{\nu,n+1}\left(w,z,t\right)=z^{1-\nu}q_{\nu,n+1}\left(w,z,t\right).
\end{split}
\]
(\ref{eq:symmetry of q^V_nu}) follows immediately from here. Then,
to establish (\ref{eq: duhamel integral equiv}), we write its right
hand side as
\[
\begin{split} & q_{\nu}\left(z,w,t\right)+\int_{0}^{t}\int_{0}^{\infty}q_{\nu}^{V}\left(z,\xi,t-\tau\right)q_{\nu}\left(\xi,w,\tau\right)V\left(\xi\right)d\xi d\tau\\
= & q_{\nu}\left(z,w,t\right)+\sum_{n=0}^{\infty}\int_{0}^{t}\int_{0}^{\infty}\tilde{q}_{\nu,n}\left(z,\xi,t-\tau\right)q_{\nu}\left(\xi,w,\tau\right)V\left(\xi\right)d\xi d\tau\\
= & q_{\nu}\left(z,w,t\right)+\sum_{n=0}^{\infty}\tilde{q}_{\nu,n+1}\left(z,w,t\right)=q_{\nu}^{V}\left(z,w,t\right),
\end{split}
\]
where again we use the fact that (\ref{eq:recursion n->n+1}) and
(\ref{eq:alternative of n->n+1}) are equivalent.

Now we move on to the second statement of Proposition \ref{prop:results on q^V_nu}.
One can apply the theory of hypoellipticity in exactly the same way
as in the proof of Proposition \ref{prop:results on model eq with potential},
to show that $\left(z,t\right)\mapsto q_{\nu}^{V}\left(z,w,t\right)$
is a smooth solution to the equation in (\ref{eq:model equation with potential}).
Then, (\ref{eq:symmetry of q^V_nu}) implies that $\left(w,t\right)\mapsto q_{\nu}^{V}\left(z,w,t\right)$
is a smooth solution to the corresponding Kolmogorov forward equation.
Again, by (\ref{eq: estimate of q^V_nu/q_nu}), for every $t>0$,
\[
\begin{split}\sup_{z\in\left(0,\infty\right)}\left|\int_{0}^{\infty}q_{v}^{V}\left(z,w,t\right)dw-\int_{0}^{\infty}q_{v}\left(z,w,t\right)dw\right| & \leq e^{t\left\Vert V\right\Vert _{u}}-1\end{split}
.
\]
So as $t\searrow0$, $\int_{0}^{\infty}q_{v}^{V}\left(z,w,t\right)dw$
tends to 1 uniformly in $z$ in any compact subset of $\left(0,\infty\right)$.
In addition, by (\ref{eq: estimate of q_nu}) and (\ref{eq:exp estimate for q^V_=00005Cnu}),
it is easy to see that for every $\delta>0$, $\lim_{t\searrow0}\int_{(0,\infty)\backslash\left(z-\delta,z+\delta\right)}q_{\nu}^{V}\left(z,w,t\right)dw=0$
uniformly for $z$ in any compact subset of $\left(0,\infty\right)$
and $\lim_{z\searrow0}q_{\nu}^{V}\left(z,w,t\right)=0$ for every
$\left(w,t\right)\in\left(0,\infty\right)^{2}$. This is sufficient
for us to conclude that $q_{\nu}^{V}\left(z,w,t\right)$ is the fundamental
solution to (\ref{eq:model equation with potential}).

Finally, to show (\ref{eq:CK for q^V_nu}), we choose any $g\in C_{c}\left(\left(0,\infty\right)\right)$
and use (\ref{eq:prob interpretation of q^V_nu}) to write
\[
\begin{split} & \int_{0}^{\infty}g\left(w\right)q_{\nu}^{V}\left(z,w,t+s\right)dw\\
= & \mathbb{E}\left[e^{\int_{0}^{t+s}V\left(Y\left(z,\tau\right)\right)d\tau}g\left(Y\left(z,t+s\right)\right);t+s<\zeta_{0}^{Y}\left(z\right)\right]\\
= & \mathbb{E}\left[e^{\int_{0}^{t}V\left(Y\left(z,\tau\right)\right)d\tau}\int_{0}^{\infty}g\left(w\right)q_{\nu}^{V}\left(Y\left(z,t\right),w,s\right)dw;t<\zeta_{0}^{Y}\left(z\right)\right]\\
= & \int_{0}^{\infty}\int_{0}^{\infty}g\left(w\right)q_{\nu}^{V}\left(\xi,w,s\right)q_{\nu}^{V}\left(z,\xi,t\right)d\xi dw,
\end{split}
\]
where again we used the strong Markov property of $Y\left(z,t\right)$.
\end{proof}

\subsection{Back to the general equation.}

After solving (\ref{eq:model equation}) and its two variations (\ref{eq:model eq with drift})
and (\ref{eq:model equation with potential}), we are now ready to
tackle the original problem (\ref{eq:general IVP equation}). Let
$\phi,$ $d$, $\psi$, $\tilde{d}$, $\theta$ and $V$ be the same
as in $\mathsection3.1$ and $\mathsection3.2$. We define 
\begin{equation}
p\left(x,y,t\right):=q_{\nu}^{V}\left(\phi\left(x\right),\phi\left(y\right),t\right)\frac{\theta\left(\phi\left(x\right)\right)}{\theta\left(\phi\left(y\right)\right)}\phi^{\prime}\left(y\right)\text{ for }\left(x,y,t\right)\in\left(0,\infty\right)^{3}.\label{eq:def of p}
\end{equation}
Compiling all the results obtained above, we state the main theorem
for $p\left(x,y,t\right)$ as follows.
\begin{thm}
\label{thm:main result on p(x,y,t)}Let $p\left(x,y,t\right)$ be
defined as in (\ref{eq:def of p}). For every $x,y>0$ and $t,s>0$,
\[
\frac{\left(\phi\left(y\right)\right)^{1-\nu}\theta^{2}\left(\phi\left(y\right)\right)}{\phi^{\prime}\left(y\right)}p\left(x,y,t\right)=\frac{\left(\phi\left(x\right)\right)^{1-\nu}\theta^{2}\left(\phi\left(x\right)\right)}{\phi^{\prime}\left(x\right)}p\left(y,x,t\right)
\]
and
\[
p\left(x,y,t+s\right)=\int_{0}^{\infty}p\left(x,u,t\right)p\left(u,y,s\right)du.
\]
For every $y>0$, $\left(x,t\right)\mapsto p\left(x,y,t\right)$ is
a smooth solution to the equation in (\ref{eq:general IVP equation}),
i.e., 
\[
\left(\partial_{t}-a\left(x\right)\partial_{x}^{2}-b\left(x\right)\partial_{x}\right)p\left(x,y,t\right)=0;
\]
for every $x>0$, $\left(y,t\right)\mapsto p\left(x,y,t\right)$ is
a smooth solution to the corresponding Kolmogorov forward equation,
i.e., 
\[
\partial_{t}p\left(x,y,t\right)-\partial_{y}^{2}\left(a\left(y\right)p\left(x,y,t\right)\right)+\partial_{y}\left(b\left(y\right)p\left(x,y,t\right)\right)=0.
\]
Moreover, $p\left(x,y,t\right)$ is the fundamental solution to (\ref{eq:general IVP equation}),
and given $f\in C_{b}\left(\left(0,\infty\right)\right)$, if 
\begin{equation}
u_{f}\left(x,t\right):=\int_{0}^{\infty}p\left(x,y,t\right)f\left(y\right)dy\text{ for }\left(x,t\right)\in\left(0,\infty\right)^{2},\label{eq:def of u_f using p(x,y,t)}
\end{equation}
then $u_{f}\left(x,t\right)$ is smooth on $\left(0,\infty\right)^{2}$
and is the unique solution in $C^{2,1}\left(\left(0,\infty\right)^{2}\right)$
to (\ref{eq:general IVP equation}). 

Let $\left\{ X\left(x,t\right):\left(x,t\right)\in\left[0,\infty\right)^{2}\right\} $
be the process defined as in (\ref{eq:transformation between X and Y}).
Then, 
\begin{equation}
u_{f}\left(x,t\right)=\mathbb{E}\left[f\left(X\left(x,t\right)\right);t<\zeta_{0}^{X}\left(x\right)\right]\text{ for every }\left(x,t\right)\in\left(0,\infty\right)^{2}.\label{eq:prob interpretation of u_f}
\end{equation}
and for every $\Gamma\subseteq\mathcal{B}\left(\left(0,\infty\right)\right)$,
\begin{equation}
\int_{\Gamma}p\left(x,y,t\right)dy=\mathbb{P}\left(X\left(x,t\right)\in\Gamma,t<\zeta_{0}^{X}\left(x\right)\right).\label{eq:prob interpretation of p(x,y,t)}
\end{equation}
\end{thm}

\begin{proof}
With all the preparations, there is not much to be done for the proof
of this theorem. All the statements on $p\left(x,y,t\right)$, except
for (\ref{eq:prob interpretation of p(x,y,t)}), follow from (\ref{eq:def of p})
and Proposition \ref{prop:results on q^V_nu} via a simple change
of variable. As for $u_{f}\left(x,t\right)$, we notice that by (\ref{eq:def of u_f using p(x,y,t)}),
\[
\begin{split}u_{f}\left(x,t\right) & =\theta\left(\phi\left(x\right)\right)\int_{0}^{\infty}q_{\nu}^{V}\left(\phi\left(x\right),w,t\right)\frac{g\left(w\right)}{\theta\left(w\right)}dw\\
 & =\theta\left(\phi\left(x\right)\right)v_{h}^{V}\left(\phi\left(x\right),t\right),
\end{split}
\]
where $g=f\circ\psi$ and $h=\frac{g}{\theta}$. Since $v_{h}^{V}\left(z,t\right)$
is smooth and is the unique solution in $C^{2,1}\left(\left(0,\infty\right)^{2}\right)$
to (\ref{eq:model equation with potential}) with initial data $h$,
$u_{f}\left(x,t\right)$ is also smooth on $\left(0,\infty\right)^{2}$
and according to Lemma \ref{lem:drift to potential transformation},
\[
u_{f}\left(x,t\right)=\tilde{v}_{g}\left(\phi\left(x\right),t\right)\text{ for }\left(x,t\right)\in\left(0,\infty\right)^{2}
\]
solves the equation (\ref{eq:model eq with drift}) with initial data
$g$. Lemma \ref{lem:change of variable} and Proposition \ref{prop:uniqueness of bar v_g  and u_f}
imply (\ref{eq:prob interpretation of u_f}) and the uniqueness of
$u_{f}\left(x,t\right)$. Finally, (\ref{eq:prob interpretation of p(x,y,t)})
follows from (\ref{eq:prob interpretation of u_f}) since $f\in C_{b}\left(\left(0,\infty\right)\right)$
is arbitrary. 
\end{proof}
Although (\ref{eq:def of p}) provides the exact formula for $p\left(x,y,t\right)$,
it is generally impossible to compute the series in (\ref{eq:def of q^V_nu})
explicitly. However, (\ref{eq:def of q^V_nu}) does offer good estimates
for $p\left(x,y,t\right)$, at least for small $t$, in terms of functions
whose exact expressions are more accessible. These estimates are more
accurate than the general heat kernel estimates such as (\ref{eq:general heat kernel estimate}).
The following facts follow immediately from (\ref{eq:estimate for q_nu,n})
and (\ref{eq: estimate of q^V_nu/q_nu}), so we will omit the proof.
\begin{cor}
\label{cor:p_approx. ratio estimate}If we define
\[
p^{approx.}\left(x,y,t\right):=q_{\nu}\left(\phi\left(x\right),\phi\left(y\right),t\right)\frac{\theta\left(\phi\left(x\right)\right)}{\theta\left(\phi\left(y\right)\right)}\phi^{\prime}\left(y\right)\text{ for }\left(x,y,t\right)\in\left(0,\infty\right)^{3},
\]
then for every $t>0$, 
\[
\sup_{x,y\in(0,\infty)^{2}}\left|\frac{p\left(x,y,t\right)}{p^{approx.}\left(x,y,t\right)}-1\right|\leq e^{\left\Vert V\right\Vert _{u}t}-1.
\]

Furthermore, for every $k\in\mathbb{N}$, if we define
\[
p^{approx.,k}\left(x,y,t\right):=\left(\sum_{n=0}^{k}q_{\nu,n}\left(\phi\left(x\right),\phi\left(y\right),t\right)\right)\frac{\theta\left(\phi\left(x\right)\right)}{\theta\left(\phi\left(y\right)\right)}\phi^{\prime}\left(y\right)\text{ for }\left(x,y,t\right)\in\left(0,\infty\right)^{3},
\]
then for every $t>0$,
\[
\sup_{x,y\in(0,\infty)^{2}}\left|\frac{p\left(x,y,t\right)-p^{approx.,k}\left(x,y,t\right)}{p^{approx.}\left(x,y,t\right)}\right|\leq e^{\left\Vert V\right\Vert _{u}t}\frac{\left\Vert V\right\Vert _{u}^{k+1}}{\left(k+1\right)!}t^{k+1}.
\]
\end{cor}

\section{Derivatives of Solutions to General Equation}

In $\mathsection2.2$, we investigated the derivatives of $q_{\nu}\left(z,w,t\right)$.
In this section, we will apply the results from $\mathsection2.2$
to studying the derivatives of $q_{\nu}^{V}\left(z,w,t\right)$. As
in the previous sections, we assume that $\nu<1$. For pedagogical
purposes, we will only discuss the derivatives of $q_{\nu}^{V}\left(z,w,t\right)$
in $z$ while $t>0$ is fixed and $z,w$ are close to 0. The derivatives
in $w$ can be treated by the symmetry of $q_{\nu}^{V}\left(z,w,t\right)$
as indicated in (\ref{eq:symmetry of q^V_nu}). Also, we will only
consider the cases when $\partial_{z}^{k}q_{\nu}\left(z,w,t\right)$,
as well as $\partial_{z}^{k}q_{\nu}^{V}\left(z,w,t\right)$, is bounded
near 0, i.e., either $\nu\in\left(-\infty,1\right)\backslash\mathbb{Z}$
and $k\in\left\{ 0,1,\cdots,\left[1-\nu\right]\right\} $, or $-\nu\in\mathbb{N}$
and $k\in\mathbb{N}$. In these cases, we are able to obtain rather
explicit bounds on $\partial_{z}^{k}q_{\nu}^{V}\left(z,w,t\right)$
for $\left(z,w\right)$ near 0. Seeing the role of $V\left(z\right)$
in configuring $q_{\nu}^{V}\left(z,w,t\right)$, one naturally expects
that the global properties of $V\left(z\right)$ will affect the regularity
of $q_{\nu}^{V}\left(z,w,t\right)$. In the upcoming discussions on
the derivatives of $q_{\nu}^{V}\left(z,w,t\right)$, we often need
to impose global conditions on $V\left(z\right)$, such as $C_{k}^{V}<\infty$.

We will start with a generalization of (\ref{eq:CK equation for q_nu}),
the Chapman-Kolmogorov equation satisfied by $q_{\nu}\left(z,w,t\right)$.
For every $\nu<1$ and $k\in\mathbb{N}$, let $Q_{\nu+k}\left(z,w,t\right)$
be defined as in (\ref{eq:def of Q_nu}).
\begin{lem}
\label{lem:formula for q_nu,(k,l)}If, either $\nu\in\left(-\infty,1\right)\backslash\mathbb{Z}$
and $k,l\in\mathbb{N}$ with $l\leq k\leq\left[1-\nu\right]$, or
$-\nu\in\mathbb{N}$ and $k,l\in\mathbb{N}$ with $l\leq k$, then
for every $z,w>0$ and every $t,s>0$, 
\begin{equation}
\int_{0}^{\infty}Q_{\nu+k}\left(z,\xi,t\right)Q_{\nu+l}\left(\xi,w,s\right)d\xi=\frac{1}{\left(t+s\right)^{k-l}}\sum_{j=0}^{k-l}\binom{k-l}{j}t^{j}s^{k-l-j}Q_{\nu+l+j}\left(z,w,t+s\right).\label{eq:relation of q_nu, (k,l)}
\end{equation}
\end{lem}

\begin{proof}
For convenience, we write $Q_{\nu,\left(k,l\right)}\left(z,w,t,s\right):=\int_{0}^{\infty}Q_{\nu+k}\left(z,\xi,t\right)Q_{\nu+l}\left(\xi,w,s\right)d\xi$.
We will prove (\ref{eq:relation of q_nu, (k,l)}) by induction on
the value of $k+l$. When $k=l=0$, (\ref{eq:relation of q_nu, (k,l)})
is simply reduced to (\ref{eq:CK equation for q_nu}). In particular,
this means that there is nothing to be done when $\nu\in\left(0,1\right)$.
We only need to treat the case when $\nu\leq0$. Assume that (\ref{eq:relation of q_nu, (k,l)})
holds for all the pairs $\left(k,l\right)$ that satisfies the condition
in the statement with $k+l\leq m$ for some $m\in\mathbb{N}$ ($m\leq2\left[1-\nu\right]$
if $\nu\notin\mathbb{Z}$). Choosing any such a pair $\left(k,l\right)$,
it suffices for us to show that (\ref{eq:relation of q_nu, (k,l)})
also holds for $\left(k+1,l\right)$ and $\left(k,l+1\right)$ (provided
that, in the later case, $l+1\leq k$). First, we use (\ref{eq:(dz)^k of Q_nu+l})
and the inductive hypothesis to write $Q_{\nu,\left(k+1,l\right)}\left(z,w,t,s\right)$
as
\[
\begin{split} & t\partial_{z}Q_{\nu,\left(k,l\right)}\left(z,w,t,s\right)+Q_{\nu,\left(k,l\right)}\left(z,w,t,s\right)\\
= & \frac{t}{\left(t+s\right)^{k+1-l}}\sum_{j=0}^{k-l}\binom{k-l}{j}t^{j}s^{k-l-j}\left(Q_{\nu+l+j+1}\left(z,w,t+s\right)-Q_{\nu+l+j}\left(z,w,t+s\right)\right)\\
 & \hspace{2cm}\hspace{2cm}\hspace{0.5cm}+\frac{1}{\left(t+s\right)^{k-l}}\sum_{j=0}^{k-l}\binom{k-l}{j}t^{j}s^{k-l-j}Q_{\nu+l+j}\left(z,w,t+s\right)\\
= & \frac{1}{\left(t+s\right)^{k+1-l}}\sum_{j=1}^{k+1-l}\binom{k-l}{j-1}t^{j}s^{k+1-l-j}Q_{\nu+l+j}\left(z,w,t+s\right)\\
 & \hspace{2cm}+\frac{1}{\left(t+s\right)^{k+1-l}}\sum_{j=0}^{k-l}\binom{k-l}{j}t^{j}s^{k+1-l-j}Q_{\nu+l+j}\left(z,w,t+s\right)\\
= & \frac{1}{\left(t+s\right)^{k+1-l}}\sum_{j=0}^{k+1-l}\binom{k+1-l}{j}t^{j}s^{k+1-l-j}Q_{\nu+l+j}\left(z,w,t+s\right).
\end{split}
\]
So (\ref{eq:relation of q_nu, (k,l)}) holds for $\left(k+1,l\right)$.
Next, it is easy to check that for every admissible pair $\left(k,l\right)$
with $l+1\leq k$,
\[
\lim_{\xi\searrow0}Q_{\nu+k}\left(z,\xi,t\right)Q_{\nu+l}\left(\xi,w,s\right)=0.
\]
Therefore, again, by (\ref{eq:(dz)^k of Q_nu+l}), we have that
\[
\begin{split}Q_{\nu,\left(k,l+1\right)}\left(z,w,t,s\right) & =s\int_{0}^{\infty}Q_{\nu+k}\left(z,\xi,t\right)\partial_{\xi}Q_{\nu+l}\left(\xi,w,s\right)d\xi+Q_{\nu,\left(k,l\right)}\left(z,w,t,s\right)\\
 & =-s\int_{0}^{\infty}\partial_{\xi}Q_{\nu+k}\left(z,\xi,t\right)Q_{\nu+l}\left(\xi,w,s\right)d\xi+Q_{\nu,\left(k,l\right)}\left(z,w,t,s\right)\\
 & =s\partial_{z}Q_{\nu,\left(k-1,l\right)}\left(z,w,t,s\right)+Q_{\nu,\left(k,l\right)}\left(z,w,t,s\right),
\end{split}
\]
which, by the inductive hypothesis, is equal to
\[
\begin{split} & \frac{s}{\left(t+s\right)^{k-l}}\sum_{j=0}^{k-1-l}\binom{k-1-l}{j}t^{j}s^{k-l-1-j}\left(Q_{\nu+l+j+1}\left(z,w,t+s\right)-Q_{\nu+l+j}\left(z,w,t+s\right)\right)\\
 & \hspace{2cm}\hspace{2cm}\hspace{1cm}+\frac{1}{\left(t+s\right)^{k-l}}\sum_{j=0}^{k-l}\binom{k-l}{j}t^{j}s^{k-l-j}Q_{\nu+l+j}\left(z,w,t+s\right)\\
= & \frac{s}{\left(t+s\right)^{k-l}}\sum_{j=0}^{k-1-l}\binom{k-1-l}{j}t^{j}s^{k-l-1-j}Q_{\nu+l+j+1}\left(z,w,t+s\right)\\
 & \hspace{1cm}+\frac{1}{\left(t+s\right)^{k-l}}\sum_{j=0}^{k-1-l}\binom{k-1-l}{j}t^{j+1}s^{k-l-j-1}Q_{\nu+l+j+1}\left(z,w,t+s\right)\\
= & \frac{1}{\left(t+s\right)^{k-l-1}}\sum_{j=0}^{k-1-l}\binom{k-1-l}{j}t^{j}s^{k-l-1-j}Q_{\nu+l+j+1}\left(z,w,t+s\right).
\end{split}
\]
This confirms that (\ref{eq:relation of q_nu, (k,l)}) also holds
for $\left(k,l+1\right)$.
\end{proof}
\begin{prop}
\label{prop:derivatives of q^V_nu}Let $\left\{ q_{\nu,n}\left(z,w,t\right):n\in\mathbb{N}\right\} $
be the sequence of functions defined as in (\ref{eq:recursion n->n+1}),
or equivalently, as in (\ref{eq:alternative of n->n+1}). For every
$k\in\mathbb{N}$ , we set
\[
S_{k}\left(z,w,t\right):=\sum_{m=0}^{k}\binom{k}{m}Q_{\nu+m}\left(z,w,t\right)\text{ for }\left(z,w,t\right)\in\left(0,\infty\right)^{3}.
\]
If $C_{k}^{V}<\infty$, where either $\nu\in\left(-\infty,1\right)\backslash\mathbb{Z}$
and $k\in\left\{ 0,1,2,\cdots,\left[1-\nu\right]\right\} $, or $-\nu\in\mathbb{N}$
and $k\in\mathbb{N}$, then for every $n\in\mathbb{N}$ and every
$\left(z,w,t\right)\in\left(0,\infty\right)^{3}$,
\begin{equation}
\begin{split}\left|\partial_{z}^{k}q_{\nu,n}\left(z,w,t\right)\right| & \leq\frac{\left(3^{k}C_{k}^{V}\right)^{n}}{n!}\frac{\left[1+\left(n\wedge k\right)t\right]^{k}}{t^{k-n}}S_{k}\left(z,w,t\right),\end{split}
\label{eq:estimate of kth deriv of q^V,n}
\end{equation}
and hence 
\begin{equation}
\left|\partial_{z}^{k}q_{\nu}^{V}\left(z,w,t\right)\right|\leq\frac{\left(1+kt\right)^{k}}{t^{k}}e^{3^{k}C_{k}^{V}t}S_{k}\left(z,w,t\right),\label{eq:estimate of kth deriv of q^V}
\end{equation}
which implies that, for every $t>0$, $\partial_{z}^{k}q_{\nu}^{V}\left(z,w,t\right)$
is bounded when $z,w$ are near 0.
\end{prop}

\begin{proof}
Note that if $\left(\nu,k\right)$ is a pair as described in the statement,
then for every $m\in\left\{ 0,1,\cdots,k\right\} $, $Q_{\nu+m}\left(z,w,t\right)\geq0$
for every $\left(z,w,t\right)\in\left(0,\infty\right)^{3}$. Since
(\ref{eq:estimate of kth deriv of q^V}) follows from (\ref{eq:def of q^V_nu})
and (\ref{eq:estimate of kth deriv of q^V,n}), we only need to show
(\ref{eq:estimate of kth deriv of q^V,n}), and we will do so by induction
on $n$. When $n=0$, (\ref{eq:estimate of kth deriv of q^V,n}) is
a trivial consequence of (\ref{eq:(dz)^k of q_nu expressed as sum of q_nu+j}).
Now assume that for some $n\in\mathbb{N}$, the inequality in (\ref{eq:estimate of kth deriv of q^V,n})
holds for every $k$ satisfying the stated requirements, i.e., $k\in\left\{ 0,1,\cdots,\left[1-\nu\right]\right\} $
if $\nu\notin\mathbb{Z}$, and $k\in\mathbb{N}$ if $-\nu\in\mathbb{N}$.
we want to show that it is also the case with $n+1$. Because (\ref{eq:estimate of kth deriv of q^V,n})
is reduced to (\ref{eq:estimate for q_nu,n}) when $k=0$, we only
need to verify the inequality for $n+1$ and $k\geq1$. 

Let us first consider the case when $k\geq n+1$. Based on the inductive
hypothesis, it is easy to see that for every $w>0$ and every $0<\tau\leq t$,
$\xi\mapsto V\left(\xi\right)q_{\nu,n}\left(\xi,w,t-\tau\right)$
is at least $k$ times differentiable on $\left(0,\infty\right)$
with $C_{k}^{V\left(\cdot\right)q_{\nu,n}\left(\cdot,w,t-\tau\right)}<\infty$
and $\lim_{\xi\searrow0}\partial_{\xi}^{j}\left(V\left(\xi\right)q_{\nu,n}\left(\xi,w,t-\tau\right)\right)=0$
for $j\in\left\{ 0,1,\cdots,\left(k\wedge\left[1-\nu\right]\right)-1\right\} $.
Thus, we can use (\ref{eq:expression of (dz)^k of v_g as integral of g^(k) using Q})
and the recursion relation (\ref{eq:alternative of n->n+1}) to write
\begin{equation}
\begin{split}\partial_{z}^{k}q_{\nu,n+1}\left(z,w,t\right) & =\int_{0}^{t/2}\int_{0}^{\infty}\partial_{z}^{k}q_{\nu}\left(z,\xi,t-\tau\right)V\left(\xi\right)q_{\nu,n}\left(\xi,w,\tau\right)d\xi d\tau\\
 & \hspace{1cm}+\int_{0}^{t/2}\int_{0}^{\infty}Q_{\nu+k}\left(z,\xi,\tau\right)\partial_{\xi}^{k}\left(V\left(\xi\right)q_{\nu,n}\left(\xi,w,t-\tau\right)\right)d\xi d\tau.
\end{split}
\label{eq:(dz)^k q^V,n+1 split into two parts (n leq k)}
\end{equation}
By (\ref{eq:(dz)^k of Q_nu+l}), the first term on the right hand
side of (\ref{eq:(dz)^k q^V,n+1 split into two parts (n leq k)})
is equal to 
\[
\sum_{j=0}^{k}\binom{k}{j}\left(-1\right)^{k-j}\int_{0}^{t/2}\frac{1}{\left(t-\tau\right)^{k}}\int_{0}^{\infty}Q_{\nu+j}\left(z,\xi,t-\tau\right)V\left(\xi\right)q_{\nu,n}\left(\xi,w,\tau\right)d\xi d\tau,
\]
which, by (\ref{eq:estimate for q_nu,n}) and (\ref{eq:relation of q_nu, (k,l)}),
is bounded by 
\[
\begin{split} & \frac{\left\Vert V\right\Vert _{u}^{n+1}}{n!}\sum_{j=0}^{k}\binom{k}{j}\sum_{m=0}^{j}\binom{j}{m}Q_{\nu+m}\left(z,w,t\right)\int_{0}^{t/2}\frac{1}{\left(t-\tau\right)^{k}}\frac{\left(t-\tau\right)^{m}\tau^{j+n-m}}{t^{j}}d\tau\\
= & \frac{\left\Vert V\right\Vert _{u}^{n+1}}{n!}\sum_{m=0}^{k}\binom{k}{m}Q_{\nu+m}\left(z,w,t\right)\sum_{p=0}^{k-m}\binom{k-m}{p}\int_{0}^{t/2}\frac{\left(t-\tau\right)^{m-k}\tau^{p+n}}{t^{p+m}}d\tau\\
= & \frac{1}{t^{k-n-1}}\frac{\left\Vert V\right\Vert _{u}^{n+1}}{n!}\sum_{m=0}^{k}\binom{k}{m}Q_{\nu+m}\left(z,w,t\right)\int_{0}^{1/2}s^{n}\left(\frac{1+s}{1-s}\right)^{k-m}ds\\
\leq & \frac{3^{k}}{2^{n+1}t^{k-n-1}}\frac{\left\Vert V\right\Vert _{u}^{n+1}}{\left(n+1\right)!}S_{k}\left(z,w,t\right),
\end{split}
\]
where in the last inequality we used the fact that $\frac{1+s}{1-s}\leq3$
for $s\in\left(0,\frac{1}{2}\right)$. Following (\ref{eq:relation of q_nu, (k,l)})
and the inductive hypothesis, the second term on the right hand side
of (\ref{eq:(dz)^k q^V,n+1 split into two parts (n leq k)}) is bounded
by
\[
\begin{split} & C_{k}^{V}\sum_{j=0}^{k}\binom{k}{j}\int_{0}^{t/2}\int_{0}^{\infty}Q_{\nu+k}\left(z,\xi,\tau\right)\left|\partial_{\xi}^{k-j}q_{\nu,n}\left(\xi,w,t-\tau\right)\right|d\xi d\tau\\
\leq & \frac{3^{kn}}{n!}\left(C_{k}^{V}\right)^{n+1}\sum_{j=0}^{k}\binom{k}{j}\int_{0}^{t/2}\frac{\left(1+n\left(t-\tau\right)\right)^{k-j}}{\left(t-\tau\right)^{k-j-n}}\sum_{p=0}^{k-j}\binom{k-j}{p}\int_{0}^{\infty}Q_{\nu+k}\left(z,\xi,\tau\right)Q_{\nu+p}\left(\xi,w,t-\tau\right)d\xi d\tau\\
= & \frac{3^{kn}}{n!}\left(C_{k}^{V}\right)^{n+1}\sum_{j=0}^{k}\binom{k}{j}\sum_{p=0}^{k-j}\binom{k-j}{p}\frac{1}{t^{k-p}}\sum_{m=p}^{k}\binom{k-p}{m-p}Q_{\nu+m}\left(z,w,t\right)\\
 & \hspace{2cm}\hspace{2cm}\hspace{1cm}\hspace{1cm}\cdot\int_{0}^{t/2}\left(1+n\left(t-\tau\right)\right)^{k-j}\tau^{m-p}\left(t-\tau\right)^{j+n-m}d\tau.
\end{split}
\]
Similarly as above, by exchanging the order of summations, we can
reduce the expression above to 
\[
\begin{split} & \frac{3^{kn}}{n!}\frac{\left(C_{k}^{V}\right)^{n+1}}{t^{k-n-1}}\sum_{m=0}^{k}\binom{k}{m}Q_{\nu+m}\left(z,w,t\right)\\
 & \hspace{1cm}\cdot\int_{0}^{1/2}\left[\frac{1+s}{1-s}+t\left(n+1\right)\left(s+\frac{n}{n+1}\right)\right]^{m}\left[1+\left(n+1\right)t\left(1-s\right)\right]^{k-m}\left(1-s\right)^{n}ds\\
\leq & \frac{3^{kn}}{n!}\left(C_{k}^{V}\right)^{n+1}\sum_{m=0}^{k}\binom{k}{m}Q_{\nu+m}\left(z,w,t\right)\frac{\left(1+t\left(n+1\right)\right)^{k}}{t^{k-n-1}}\int_{0}^{1/2}\left(\frac{1+s}{1-s}\right)^{m}\left(1-s\right)^{n}ds\\
\leq & \frac{1-2^{-n-1}}{\left(n+1\right)!}\left(3^{k}C_{k}^{V}\right)^{n+1}S_{k}\left(z,w,t\right)\frac{\left(1+\left(n+1\right)t\right)^{k}}{t^{k-n-1}}
\end{split}
\]
where in the second last inequality we used the fact that $s+\frac{n}{n+1}\leq\frac{1+s}{1-s}$
for all $s\in\left(0,\frac{1}{2}\right).$ Putting the two estimates
on the right hand side of (\ref{eq:(dz)^k q^V,n+1 split into two parts (n leq k)})
together, we get that
\[
\begin{split}\left|\partial_{z}^{k}q_{\nu,n+1}\left(z,w,t\right)\right| & \leq\frac{\left(3^{k}C_{k}^{V}\right)^{n+1}}{\left(n+1\right)!}S_{k}\left(z,w,t\right)\frac{\left(1+\left(n+1\right)t\right)^{k}}{t^{k-n-1}}.\end{split}
\]
which confirms that (\ref{eq:estimate of kth deriv of q^V,n}) holds
for $n+1$ and $k\geq n+1$. 

Now assume that $1\leq k\leq n$. This time we will switch to the
recursion relation (\ref{eq:recursion n->n+1}). By (\ref{eq:recursion n->n+1}),
(\ref{eq:relation of q_nu, (k,l)}) and the inductive hypothesis,
we have that $\left|\partial_{z}^{k}q_{\nu,n+1}\left(z,w,t\right)\right|$
is bounded by
\begin{align*}
 & \left\Vert V\right\Vert _{u}\int_{0}^{t}\int_{0}^{\infty}\left|\partial_{z}^{k}q_{\nu,n}\left(z,\xi,t-\tau\right)\right|q_{\nu}\left(\xi,w,\tau\right)d\xi d\tau\\
\leq & \frac{3^{kn}\left(C_{k}^{V}\right)^{n+1}\left(1+kt\right)^{k}}{n!}\sum_{m=0}^{k}\binom{k}{m}\sum_{p=0}^{k-m}\binom{k-m}{p}\frac{Q_{\nu+m}\left(z,w,t\right)}{t^{p+m}}\int_{0}^{t}\left(t-\tau\right)^{n-k+m}\tau^{p}d\tau\\
= & \frac{3^{kn}\left(C_{k}^{V}\right)^{n+1}\left(1+kt\right)^{k}}{n!t^{k-n-1}}\sum_{m=0}^{k}\binom{k}{m}Q_{\nu+m}\left(z,w,t\right)\int_{0}^{1}\left(1-s\right)^{n-k+m}\left(1+s\right)^{k-m}ds\\
\le & \frac{\left(3^{k}C_{k}^{V}\right)^{n+1}\left(1+kt\right)^{k}}{\left(n+1\right)!t^{k-n-1}}S_{k}\left(z,w,t\right),
\end{align*}
where in the last inequality we computed that
\begin{align*}
\int_{0}^{1}\left(1-s\right)^{n-k+m}\left(1+s\right)^{k-m}ds & =\frac{1}{n+1}+\frac{2\left(k-m\right)}{n+1}\int_{0}^{1}\left(1-s\right)^{n-k+m}\left(1+s\right)^{k-m-1}ds\\
 & \leq\frac{1}{n+1}\left(1+2\left(k-m\right)\int_{0}^{1}\left(1+s\right)^{k-1}ds\right)\\
 & \leq\frac{2^{k+1}-1}{n+1}\leq\frac{3^{k}}{n+1}\text{ for every }k\geq1.
\end{align*}
\end{proof}
\begin{rem}
Here we only provide the detailed treatment of the derivatives of
$q_{\nu}^{V}\left(z,w,t\right)$ in $z$. We can obtain estimates
on the derivatives of $q_{\nu}^{V}\left(z,w,t\right)$ in $w$ by
(\ref{eq:symmetry of q^V_nu}), the symmetry property of $q_{\nu}^{V}\left(z,w,t\right)$.
Alternatively, we can rely on the counterpart of (\ref{eq:expression of (dz)^k of v_g as integral of g^(k)})
with respect to the forward variable $w$. Namely, if $g\in C^{l}\left(\left(0,\infty\right)\right)$
with $C_{l}^{g}<\infty$ for some $l\in\mathbb{N}$, and 
\[
v_{g}^{*}\left(w,t\right):=\int_{0}^{\infty}q_{\nu}\left(z,w,t\right)g\left(z\right)dz\text{ for }\left(w,t\right)\in\left(0,\infty\right)^{2},
\]
then, by (\ref{eq:(dz)^k q_nu =00003D (dw)^k q_(nu+k)}), we have
that 
\[
\partial_{w}^{l}v_{g}^{*}\left(w,t\right)=\int_{0}^{\infty}q_{\nu-l}\left(z,w,t\right)g^{\left(l\right)}\left(z\right)dz\text{ for every }\left(w,t\right)\in\left(0,\infty\right)^{2}.
\]
Thus, one can mimic the proof of Proposition \ref{prop:derivatives of q^V_nu}
and apply the formula above, in a similar way as we used (\ref{eq:expression of (dz)^k of v_g when nu=00003D-N}),
to studying the derivatives in the forward variable whenever the time
variable is small. With this method, not only can we obtain estimates
on the derivatives of $q_{\nu}^{V}\left(z,w,t\right)$ in $w$, we
can also treat the mixed derivatives of $q_{\nu}^{V}\left(z,w,t\right)$
in $z$ and $w$. Because it is largely a repetition of the proof
of Proposition \ref{prop:derivatives of q^V_nu}, possibly with more
cumbersome technicalities, we will not carry out the derivations in
details.

It is also possible to use our method to study the case when $\partial_{z}^{k}q_{\nu}^{V}\left(z,w,t\right)$
is unbounded near 0, i.e., when $\nu\in\left(-\infty,1\right)\backslash\mathbb{Z}$
and $k\geq\left[1-\nu\right]+1$. However, in this case we face the
obstacle that $Q_{\nu+k}\left(z,w,t\right)$ is possibly negative
and/or locally non-integrable in $z$ near 0. Therefore, it is difficult
to derive the counterpart of (\ref{eq:relation of q_nu, (k,l)}) and
(\ref{eq:estimate of kth deriv of q^V,n}). Besides, we also expect
that the asymptotics of $V\left(z\right)$ near 0 will play a more
significant role in determining the regularity/singularity level of
$q_{\nu}^{V}\left(z,w,t\right)$ near 0. We plan to return to this
problem in the sequel to this paper in which we will investigate the
behaviors of $q_{\nu}^{V}\left(z,w,t\right)$ near 0 under proper
local conditions on $V\left(z\right)$.
\end{rem}

Proposition \ref{prop:derivatives of q^V_nu} leads to the following
results on the derivatives of $v_{h}^{V}\left(z,t\right)$.
\begin{cor}
\label{cor:asymp of (dz)^k v^V_h}Assume that $C_{k}^{V}<\infty$,
where either $\nu\in\left(-\infty,1\right)\backslash\mathbb{Z}$ and
$k\in\left\{ 0,1,2,\cdots,\left[1-\nu\right]\right\} $, or $-\nu\in\mathbb{N}$
and $k\in\mathbb{N}$, Given a function $h$ on $\left(0,\infty\right)$
that satisfies (\ref{eq:estimate on initial h}), let $v_{h}^{V}\left(z,t\right)$
be defined as in (\ref{eq:def of v^V_h}). Then for every $t>0$,
$\partial_{z}^{k}v_{h}^{V}\left(z,t\right)$ is bounded near 0, and
for every $M>0$,
\[
\sup_{\left(z,t\right)\in\left(0,M\right)\times\left(0,1\right)}t^{k}\left|\partial_{z}^{k}v_{h}^{V}\left(z,t\right)\right|<\infty.
\]
\end{cor}

It is easy to see that (\ref{eq:estimate of kth deriv of q^V}) enables
us to compute $\partial_{z}^{k}v_{h}^{V}\left(z,t\right)$ by differentiating
under the integral sign in the right hand side of (\ref{eq:def of v^V_h}),
from where the results in Corollary \ref{cor:asymp of (dz)^k v^V_h}
follow in a straightforward way. The proof is omitted. 

Proposition \ref{prop:derivatives of q^V_nu} also provides a passage
to the smoothness of $q_{\nu}^{V}\left(z,w,t\right)$ in all three
variables. 
\begin{cor}
\label{cor:smoothness of q^V_nu}If $C_{1}^{V}<\infty$, then $q_{\nu}^{V}\left(z,w,t\right)$
is smooth in $\left(z,w,t\right)$ on $\left(0,\infty\right)^{3}$
.
\end{cor}

\begin{proof}
The first step is to show that, for every $\nu<1$, 
\begin{equation}
\left|\partial_{z}q_{\nu}^{V}\left(z,w,t\right)\right|\leq\frac{\left(1+t\right)}{t}e^{3C_{1}^{V}t}\left(q_{\nu}\left(z,w,t\right)+q_{\nu+1}\left(z,w,t\right)\right)\text{ for every }\left(z,w,t\right)\in\left(0,\infty\right)^{3}.\label{eq:estimate of (dz)q^V_nu}
\end{equation}
Obviously, if $\nu\leq0$, then $\left[1-\nu\right]\geq1$ and (\ref{eq:estimate of (dz)q^V_nu})
is just (\ref{eq:estimate of kth deriv of q^V}) with $k=1$. Now
we assume that $0<\nu<1$ and observe that, in this case, $Q_{\nu+1}\left(z,w,t\right)=q_{\nu+1}\left(z,w,t\right)$
is always non-negative and $z\mapsto q_{\nu+1}\left(z,w,t\right)$
is locally integrable near 0. It is easy to see that, if we restrict
ourselves to the case $k=1$, then we can repeat the entire proof
of Lemma \ref{lem:formula for q_nu,(k,l)} and Proposition \ref{prop:derivatives of q^V_nu}
without any change and get the same estimate on $\partial_{z}q_{\nu}^{V}\left(z,w,t\right)$,
which is exactly (\ref{eq:estimate of (dz)q^V_nu}). Next, following
(\ref{eq: estimate of q_nu}), (\ref{eq:estimate of q_nu for nu>=00003D1}),
(\ref{eq:(dz)^k of Q_nu+l}) and (\ref{eq: duhamel integral equiv}),
we see that for every $z>0$, $\left(w,t\right)\mapsto\partial_{z}q_{\nu}^{V}\left(z,w,t\right)$
is smooth on $\left(0,\infty\right)^{2}$. Since $\left(z,t\right)\mapsto q_{\nu}^{V}\left(z,w,t\right)$
and $\left(w,t\right)\mapsto q_{\nu}^{V}\left(z,w,t\right)$ are smooth
solutions to the backward equation and, respectively, the forward
equation associated with (\ref{eq:model equation with potential}),
we conclude that $\left(z,w,t\right)\mapsto q_{\nu}^{V}\left(z,w,t\right)$
is smooth on $\left(0,\infty\right)^{3}$.
\end{proof}
Now we return to $p\left(x,y,t\right)$, the fundamental solution
to (\ref{eq:general IVP equation}), and investigate the derivative
of $p\left(x,y,t\right)$ in $x$ near the boundary 0. By (\ref{eq:def of p}),
Corollary \ref{cor:smoothness of q^V_nu} immediately implies that,
if $C_{1}^{V}<\infty$, then $\left(x,y,t\right)\mapsto p\left(x,y,t\right)$
is smooth on $\left(0,\infty\right)^{3}$. We also want to obtain
specific bounds on the derivatives of $p\left(x,y,t\right)$ in $x$.
It is apparent that, besides the derivatives of $q_{\nu}^{V}\left(z,w,t\right)$,
the transformations $\phi$ and $\theta$ will also affect the regularity
of $p\left(x,y,t\right)$, which makes it complicated to track down
the derivatives of $p\left(x,y,t\right)$. By Faà di Bruno's formula,
we have that, for every $k\in\mathbb{N}$, $\partial_{x}^{k}p\left(x,y,t\right)$
is equal to
\begin{equation}
\begin{split} & \frac{\phi^{\prime}\left(y\right)}{\theta\left(\phi\left(y\right)\right)}\sum_{j=1}^{k}\sum_{i=0}^{j}\binom{j}{i}\partial_{z}^{i}q_{\nu}^{V}\left(\phi\left(x\right),\phi\left(y\right),t\right)\\
 & \quad\hspace{2cm}\cdot\theta^{\left(j-i\right)}\left(\phi\left(x\right)\right)B_{k,j}\left(\phi^{\prime}\left(x\right),\phi^{\prime\prime}\left(x\right),\cdots,\phi^{\left(k-j+1\right)}\left(x\right)\right),
\end{split}
\label{eq:formula of (dx)^k p(x,y,t)}
\end{equation}
where $B_{k,j}$, $j=1,\cdots,k$, refer to the Bell polynomials,
i.e., 
\[
B_{k,j}\left(x_{1},x_{2},\cdots,x_{k-j+1}\right):=\sum\frac{k!}{i_{1}!i_{2}!\cdots i_{k-j+1}!}\prod_{p=1}^{k-j+1}\left(\frac{x_{p}}{p!}\right)^{i_{p}},\,x_{p}\in\mathbb{R}\text{ for }1\leq p\leq k-j+1,
\]
with the summation taken over the collection of $\left(i_{1},\cdots,i_{k-j+1}\right)\subseteq\mathbb{N}^{k-j+1}$
such that $\sum_{p=1}^{k-j+1}i_{p}=j$ and $\sum_{p=1}^{k-j+1}p\cdot i_{p}=k$.
Seeing from (\ref{eq:formula of (dx)^k p(x,y,t)}), it is clear that
in general we cannot directly compare the regularity of $p\left(x,y,t\right)$
with that of $q_{\nu}^{V}\left(z,w,t\right)$, unless extra conditions
are imposed on $\theta$, $\phi$ and $V$.
\begin{prop}
\label{prop:regularity of p(x,y,t) and u_f}Assume that $a\left(x\right)$
and $b\left(x\right)$ satisfy Condition 1-3 and $p\left(x,y,t\right)$
is defined as in (\ref{eq:def of p}). Let $\nu$ and $k$ be such
that either $\nu\in\left(-\infty,1\right)\backslash\mathbb{Z}$ and
$k\in\left\{ 0,1,\cdots,\left[1-\nu\right]\right\} $, or $-\nu\in\mathbb{N}$
and $k\in\mathbb{N}$. Suppose that $C_{k}^{V}<\infty$, and $\phi$
and $\theta$ have bounded derivatives to the order of $k$ in a neighborhood
of 0. For $M>0$, we set 
\[
C_{k,\left(0,\phi\left(M\right)\right)}^{\theta}:=\max_{j=0,1,2,\cdots,k}\sup_{z\in\left(0,\phi\left(M\right)\right)}\left|\theta^{\left(j\right)}\left(z\right)\right|\text{ and }C_{k,\left(0,M\right)}^{\phi}:=\max_{j=0,1,2,\cdots,k}\sup_{x\in\left(0,M\right)}\left|\phi^{\left(j\right)}\left(x\right)\right|.
\]
Then for every $\left(x,y,t\right)\in\left(0,M\right)^{2}\times\left(0,\infty\right)$,
\[
\left|\partial_{x}^{k}p\left(x,y,t\right)\right|\leq C_{k,\left(0,\phi\left(M\right)\right)}^{\theta}\mathfrak{T}_{k}\left(C_{k,\left(0,M\right)}^{\phi}\right)e^{3^{k}C_{k}^{V}t}\left(\frac{1}{t}+k+1\right)^{k}S_{k}\left(\phi\left(x\right),\phi\left(y\right),t\right)\frac{\left|\phi^{\prime}\left(y\right)\right|}{\theta\left(\phi\left(y\right)\right)},
\]
where $\mathfrak{T}_{k}$ refers to the $k$th Touchard polynomial,
i.e.,
\[
\mathfrak{T}_{k}\left(x\right):=\sum_{j=1}^{k}B_{k,j}\left(x,\cdots,x\right)\text{ for }x\in\mathbb{R}.
\]
In particular, for every $t>0$, $\partial_{x}^{k}p\left(x,y,t\right)$
is bounded when $x,y$ are near 0. 

Moreover, given $f\in C_{b}\left(\left(0,\infty\right)\right)$, if
$u_{f}\left(x,t\right)$ is defined as in (\ref{eq:def of u_f using p(x,y,t)}),
then for every $t>0$, $\partial_{x}^{k}u_{f}\left(x,t\right)$ is
bounded near 0, and 
\[
\sup_{\left(x,t\right)\in\left(0,M\right)\times\left(0,1\right)}t^{k}\left|\partial_{x}^{k}u_{f}\left(x,t\right)\right|<\infty.
\]
 
\end{prop}

\begin{proof}
To prove the estimate on $\partial_{x}^{k}p\left(x,y,t\right)$, we
derive from (\ref{eq:estimate of kth deriv of q^V}) and (\ref{eq:formula of (dx)^k p(x,y,t)})
that $\left|\partial_{x}^{k}p\left(x,y,t\right)\right|$ is bounded
by
\[
\begin{split} & C_{k,\left(0,\phi\left(M\right)\right)}^{\theta}\frac{\left|\phi^{\prime}\left(y\right)\right|}{\theta\left(\phi\left(y\right)\right)}\sum_{j=1}^{k}\sum_{i=0}^{j}\binom{j}{i}\left|\partial_{z}^{i}q_{\nu}^{V}\left(\phi\left(x\right),\phi\left(y\right),t\right)\right|\left|B_{k,j}\left(\phi^{\prime}\left(x\right),\phi^{\prime\prime}\left(x\right),\cdots,\phi^{\left(k-j+1\right)}\left(x\right)\right)\right|\\
\leq & C_{k,\left(0,\phi\left(M\right)\right)}^{\theta}\frac{\left|\phi^{\prime}\left(y\right)\right|}{\theta\left(\phi\left(y\right)\right)}\sum_{j=1}^{k}B_{k,j}\left(C_{k,\left(0,M\right)}^{\phi},C_{k,\left(0,M\right)}^{\phi},\cdots,C_{k,\left(0,M\right)}^{\phi}\right)\sum_{i=0}^{k}\binom{k}{i}\left|\partial_{z}^{i}q_{\nu}^{V}\left(\phi\left(x\right),\phi\left(y\right),t\right)\right|\\
\leq & C_{k,\left(0,\phi\left(M\right)\right)}^{\theta}e^{3^{k}C_{k}^{V}t}\frac{\left|\phi^{\prime}\left(y\right)\right|}{\theta\left(\phi\left(y\right)\right)}\mathfrak{T}_{k}\left(C_{k,\left(0,M\right)}^{\phi}\right)\sum_{i=0}^{k}\binom{k}{i}\frac{\left(1+kt\right)^{i}}{t^{i}}\sum_{m=0}^{i}\binom{i}{m}Q_{\nu+m}\left(\phi\left(x\right),\phi\left(y,t\right)\right)\\
= & C_{k,\left(0,M\right)}^{\theta}e^{3^{k}C_{k}^{V}t}\frac{\left|\phi^{\prime}\left(y\right)\right|}{\theta\left(\phi\left(y\right)\right)}\mathfrak{T}_{k}\left(C_{k,\left(0,M\right)}^{\phi}\right)\sum_{m=0}^{k}\binom{k}{m}Q_{\nu+m}\left(\phi\left(x\right),\phi\left(y,t\right)\right)\sum_{p=0}^{k-m}\binom{k-m}{p}\frac{\left(1+kt\right)^{p+m}}{t^{p+m}}\\
\leq & C_{k,\left(0,\phi\left(M\right)\right)}^{\theta}e^{3^{k}C_{k}^{V}t}\frac{\left|\phi^{\prime}\left(y\right)\right|}{\theta\left(\phi\left(y\right)\right)}\mathfrak{T}_{k}\left(C_{k,\left(0,M\right)}^{\phi}\right)S_{k}\left(\phi\left(x\right),\phi\left(y,t\right)\right)\frac{\left(1+\left(k+1\right)t\right)^{k}}{t^{k}}
\end{split}
\]

As for the last statement, according to (\ref{eq:def of p}), we have
that $u_{f}\left(x,t\right)=v_{h}^{V}\left(\phi\left(x\right),t\right)\theta\left(\phi\left(x\right)\right)$
with $h:=\frac{f\circ\psi}{\theta}$. Using Faà di Bruno's formula
again, we have that
\[
\begin{split}\partial_{x}^{k}u_{f}\left(x,t\right) & =\sum_{j=1}^{k}\sum_{i=0}^{j}\binom{j}{i}\partial_{z}^{i}v_{h}^{V}\left(\phi\left(x\right),t\right)\theta^{\left(j-i\right)}\left(\phi\left(x\right)\right)\cdot B_{k,j}\left(\phi^{\prime}\left(x\right),\phi^{\prime\prime}\left(x\right),\cdots,\phi^{\left(k-j+1\right)}\left(x\right)\right).\end{split}
\]
Therefore, the desired conclusion follows directly from Corollary
\ref{cor:asymp of (dz)^k v^V_h}.
\end{proof}

\section{Examples and Further Questions}

The framework laid out in the previous sections allows us to treat
a wide range of choices of $a\left(x\right)$ and $b\left(x\right)$.
In particular, we will revisit the examples introduced in $\mathsection1.2$,
which we now can solve explicitly with the tools developed in the
previous sections. At the end of the article, we will raise further
questions regarding the ``fitness'' of the global conditions imposed
on $a\left(x\right)$, $b\left(x\right)$ and $V\left(z\right)$. 

\subsection{Examples with $a\left(x\right)=x^{\alpha}$ for $\alpha\in\left(0,2\right)$.}

It should be clear from the preceding discussions that the potential
function $V\left(z\right)$ is the main factor that prevents us from
finding explicit expressions for $q_{\nu}^{V}\left(z,w,t\right)$.
In the case when $V\left(z\right)$ is trivial, e.g., when $V\left(z\right)\equiv0$,
the exact formulas of $p\left(x,y,t\right)$ is readily available.
There are many choices of $a\left(x\right)$ and $b\left(x\right)$
that will reduce $V\left(z\right)$ to zero. Here we will investigate
one notable case when $a\left(x\right)=x^{\alpha}$ for some $\alpha\in\left(0,2\right)$
and $b\left(x\right)\equiv0$. Namely, let us consider, for some $f\in C_{b}\left(\left(0,\infty\right)\right)$,
\begin{equation}
\begin{array}{c}
\partial_{t}u_{f}\left(x,t\right)=x^{\alpha}\partial_{x}^{2}u_{f}\left(x,t\right)\text{ for }\left(x,t\right)\in\left(0,\infty\right)^{2},\\
\lim_{t\searrow0}u_{f}\left(x,t\right)=f\left(x\right)\text{ for }x\in\left(0,\infty\right)\text{ and }\lim_{x\searrow0}u_{f}\left(x,t\right)=0\text{ for }t\in\left(0,\infty\right),
\end{array}\label{eq:IVP with x^(alpha)}
\end{equation}
Let $p_{\alpha}\left(x,y,t\right)$ be the fundamental solution to
(\ref{eq:IVP with x^(alpha)}). We check immediately that for every
$x>0$,
\[
\phi\left(x\right)=\frac{1}{4}\left(\int_{0}^{x}\frac{ds}{s^{\frac{\alpha}{2}}}\right)^{2}=\frac{x^{2-\alpha}}{\left(2-\alpha\right)^{2}}\text{ and }d\left(x\right)=\frac{-\alpha x^{\alpha-1}}{4x^{\frac{\alpha}{2}}}\int_{0}^{x}\frac{ds}{s^{\frac{\alpha}{2}}}+\frac{1}{2}-\nu\equiv0
\]
provided that $\nu=\frac{1-\alpha}{2-\alpha}$. It is clear that the
choices of $\alpha$, $a\left(x\right)$ and $b\left(x\right)$ satisfy
Condition 1-3 trivially and $V\left(z\right)\equiv0$. By Proposition
\ref{prop: q_nu is the fundamental solution} and Theorem \ref{thm:main result on p(x,y,t)},
we have that for every $\left(x,y,t\right)\in\left(0,\infty\right)^{3}$,
\begin{equation}
\begin{split}p_{\alpha}\left(x,y,t\right) & =\frac{x^{\frac{1}{2}}y^{\frac{1}{2}-\alpha}}{t\left(2-\alpha\right)}e^{-\frac{x^{2-\alpha}+y^{2-\alpha}}{\left(2-\alpha\right)^{2}t}}I_{\frac{1}{2-\alpha}}\left(\frac{2}{\left(2-\alpha\right)^{2}}\frac{\left(xy\right)^{1-\frac{\alpha}{2}}}{t}\right)\\
 & =\frac{xy^{1-\alpha}}{t^{\frac{3-\alpha}{2-\alpha}}\left(2-\alpha\right)^{\frac{4-\alpha}{2-\alpha}}}e^{-\frac{x^{2-\alpha}+y^{2-\alpha}}{\left(2-\alpha\right)^{2}t}}\sum_{n=0}^{\infty}\frac{\left(xy\right)^{n\left(2-\alpha\right)}}{t^{2n}\left(2-\alpha\right)^{4n}n!\Gamma\left(n+\frac{3-\alpha}{2-\alpha}\right)}.
\end{split}
\label{eq:def of p_alpha}
\end{equation}
When $\alpha=1$, (\ref{eq:def of p_alpha}) is reduced to (\ref{eq:q_0 of WF model equation}),
as we have expected. 

Based on (\ref{eq:def of p_alpha}) and Proposition \ref{prop: q_nu is the fundamental solution},
we have the following facts.
\begin{prop}
$p_{\alpha}\left(x,y,t\right)$\textbf{ }is smooth on $\left(0,\infty\right)^{3}$
and is the fundamental solution to (\ref{eq:IVP with x^(alpha)}).
For every $t>0$, $\partial_{x}^{k}p_{\alpha}\left(x,y,t\right)$
is bounded when $x,y$ are near 0 for $k=1$ if $\alpha\in\left(1,2\right)$,
for $k\leq2$ if $\alpha\in\left(0,1\right)$, and for every $k\in\mathbb{N}$
if $\alpha=1$.
\end{prop}

As we have mentioned in $\mathsection1.2$ that when $\alpha\in\left(0,2\right)$,
with respect to the diffusion associated with (\ref{eq:IVP with x^(alpha)}),
the boundary 0 is either a regular or an exit boundary, which implies
that 0 is attainable and there will be ``mass loss'' through 0 as
soon as $t>0$. On the other hand, as we have seen in $\mathsection1.3$
that when $\alpha=2$, the underlying diffusion process associated
with $x^{2}\partial_{x}^{2}$ will almost surely never hit the boundary
0, which results in 0 being unattainable and ``no mass loss'' through
0 for any $t>0$. Now we have the convenient tools to compute the
amount of mass loss at any time and investigate the transition of
the attainability of 0 quantitatively with respect to $\alpha$ as
$\alpha\nearrow2$.
\begin{cor}
For every $\alpha\in\left(0,2\right)$, we set
\[
m_{\alpha}\left(x,t\right):=1-\int_{0}^{\infty}p_{\alpha}\left(x,y,t\right)dy\text{ for }\left(x,t\right)\in\left(0,\infty\right)^{2}.
\]
Then for every $\left(x,t\right)\in\left(0,\infty\right)^{2}$, as
$\alpha\nearrow2$,
\[
m_{\alpha}\left(x,t\right)=\frac{x^{\alpha-1}e^{-\frac{x^{2-\alpha}}{\left(2-\alpha\right)^{2}t}}}{\Gamma\left(\frac{1}{2-\alpha}\right)\left(\left(2-\alpha\right)^{2}t\right)^{\frac{\alpha-1}{2-\alpha}}}\left(1+\frac{t}{x^{2-\alpha}}\left(2-\alpha\right)+\mathcal{O}\left(\left(2-\alpha\right)^{2}\right)\right),
\]
which implies that 
\[
\lim_{\alpha\nearrow2}\frac{-\ln m_{\alpha}\left(x,t\right)}{\frac{x^{2-\alpha}}{\left(2-\alpha\right)^{2}t}}=1
\]
and the convergence is uniformly fast in $\left(x,t\right)$ in any
compact subset of $\left(0,\infty\right)^{2}$.
\end{cor}

\begin{proof}
The second statement follows immediately from the first statement
and Stirling's approximation, so we only need to show the first statement.
By (\ref{eq:def of p_alpha}), we have that 
\[
m_{\alpha}\left(x,t\right)=1-\int_{0}^{\infty}p_{\alpha}\left(x,y,t\right)dy=1-\int_{0}^{\infty}q_{\frac{1-\alpha}{2-\alpha}}\left(\frac{x^{2-\alpha}}{\left(2-\alpha\right)^{2}},w,t\right)dw,
\]
which, by (\ref{eq:total mass of q_nu}), is equal to $\frac{1}{\Gamma\left(\frac{1}{2-\alpha}\right)}\int_{\frac{x^{\left(2-\alpha\right)}}{\left(2-\alpha\right)^{2}t}}^{\infty}s^{-\frac{1-\alpha}{2-\alpha}}e^{-s}ds$.
Set $\eta:=\frac{1}{2-\alpha}$ and $T:=\frac{x^{\left(2-\alpha\right)}}{\left(2-\alpha\right)^{2}t}$.
Then, 
\[
\begin{split}m_{\alpha}\left(x,t\right)=\frac{1}{\Gamma\left(\eta\right)}\int_{T}^{\infty}s^{\eta-1}e^{-s}ds & =\frac{T^{\eta-1}e^{-T}}{\Gamma\left(\eta\right)}\left[1+\frac{\eta-1}{T}\int_{0}^{\infty}\left(1+\frac{\tau}{T}\right)^{\eta-2}e^{-\tau}d\tau\right].\end{split}
\]
To complete the proof, it is sufficient to notice that, when $\eta$
is large,
\[
\frac{\eta-1}{T}=\frac{t}{x^{\frac{1}{\eta}}}\frac{1}{\eta}+\mathcal{O}\left(\frac{1}{\eta^{2}}\right)\text{ and }\int_{0}^{\infty}\left(1+\frac{\tau}{T}\right)^{\eta-2}e^{-\tau}d\tau=1+\mathcal{O}\left(\frac{1}{\eta}\right).
\]
\end{proof}
\begin{rem}
From the previous example we notice that $q_{\nu}\left(z,w,t\right)$
and $p_{\alpha}\left(x,y,t\right)$ exhibit different levels of regularity
near 0, even under identical boundary classification and boundary
condition. Putting the special case $\nu=0$ (equivalently, $\alpha=1$)
on the side, we see that when 0 is an exit boundary (i.e., $\nu<0$
and $\alpha\in\left(1,2\right)$), compared with $p_{\alpha}\left(x,y,t\right)$,
$q_{\nu}\left(z,w,t\right)$ has as many as or more orders of bounded
derivatives in the backward variable near 0; when 0 is a regular boundary
(i.e., $\nu\in\left(0,1\right)$ and $\alpha\in\left(0,1\right)$),
$\partial_{z}q_{\nu}\left(z,w,t\right)$ blows up as $z$ tends to
0, but $\partial_{x}^{2}p_{\alpha}\left(x,y,t\right)$ stays bounded
all the way to 0. Technically speaking, this difference in the regularity
level between $q_{\nu}\left(z,w,t\right)$ and $p_{\alpha}\left(x,y,t\right)$
is caused by the change of variable in (\ref{eq:def of p_alpha}).
Heuristically speaking, we may be able to predict this difference
if we view these derivatives as ``indicators'' of how sensitive
the hitting distribution of the underlying process is with respect
to the backward variable. Assume that $z$ and $x$ are close to 0.
In the scenario of exit boundary, for $q_{\nu}\left(z,w,t\right)$,
the constant negative drift $\nu\partial_{z}$ overtakes the diffusion
coefficient in controlling the diffusion process and makes the hitting
distribution less sensitive to $z$; while for $p_{\alpha}\left(x,y,t\right)$,
the relatively high degeneracy in $x^{\alpha}\partial_{x}^{2}$ makes
the hitting distribution more sensitive to $x$. In the scenario of
regular boundary, for $q_{\nu}\left(z,w,t\right)$, the drift is now
positive and hence the diffusion coefficient $z\partial_{z}^{2}$
has a bigger impact on the hitting distribution; since $x^{\alpha}\partial_{x}^{2}$
is less degenerate than $z\partial_{z}^{2}$ in this scenario, the
hitting distribution associated with $p_{\alpha}\left(x,y,t\right)$
is less sensitive in $x$ than that with $q_{\nu}\left(z,w,t\right)$
in $z$. 
\end{rem}

We can also consider variations of $L=x^{\alpha}\partial_{x}^{2}$,
e.g., the one with a drift coefficient that vanishes at the boundary
$0$ following another power of $x$. Assume that $\beta\geq$1 is
a constant, and $\varphi\in C^{\infty}\left(\left(0,\infty\right)\right)$
is such that $C_{k}^{\varphi}<\infty$ for every $k\in\mathbb{N}$,
$\varphi$ decays faster than any polynomial at infinity, and $\lim_{x\searrow0}\varphi\left(x\right)\neq0$.
We consider the operator $L=x^{\alpha}\partial_{x}^{2}+x^{\beta}\varphi\left(x\right)\partial_{x}$
and the associated boundary/initial value problem 
\begin{equation}
\begin{array}{c}
\partial_{t}u_{f}\left(x,t\right)=x^{\alpha}\partial_{x}^{2}u_{f}\left(x,t\right)+x^{\beta}\varphi\left(x\right)\partial_{x}u_{f}\left(x,t\right)\text{ for }\left(x,t\right)\in\left(0,\infty\right)^{2},\\
\lim_{t\searrow0}u_{f}\left(x,t\right)=f\left(x\right)\text{ for }x\in\left(0,\infty\right)\text{ and }\lim_{x\searrow0}u_{f}\left(x,t\right)=0\text{ for }t\in\left(0,\infty\right),
\end{array}\label{eq:IVP with x^(alpha) and drift}
\end{equation}
for some $f\in C_{b}\left(\left(0,\infty\right)\right)$. Let $p\left(x,y,t\right)$
be the fundamental solution to (\ref{eq:IVP with x^(alpha) and drift}).
Same as above, we determine that
\[
\phi\left(x\right)=\frac{x^{2-\alpha}}{\left(2-\alpha\right)^{2}}\text{ and }d\left(x\right)=\frac{x^{\beta+1-\alpha}\varphi\left(x\right)}{2-\alpha}\text{ for every }x>0,
\]
and $\nu=\frac{1-\alpha}{2-\alpha}$. Given the choice of $\beta$
and $\varphi\left(x\right)$, it is easy to verify that Condition
1-3 are satisfied. Therefore, according to Theorem \ref{thm:main result on p(x,y,t)}
and Corollary \ref{cor:p_approx. ratio estimate}, we have the following
result.
\begin{prop}
\label{prop:p(x,y,t) of x^alapha(dx)^2 and a drift}We set 
\[
\varLambda\left(x\right):=-\frac{x^{2\beta-\alpha}\varphi^{2}\left(x\right)}{4}-\frac{\left(\beta-\alpha\right)x^{\beta-1}\varphi\left(x\right)}{2}-\frac{x^{\beta}\varphi^{\prime}\left(x\right)}{2}\text{ for }x>0,
\]
and, for every $\left(x,y,t\right)\in\left(0,\infty\right)^{3}$,
define $p_{\alpha,0}\left(x,y,t\right):=p_{\alpha}\left(x,y,t\right)e^{\frac{1}{2}\int_{x}^{y}u^{\beta-\alpha}\varphi\left(u\right)du}$
and recursively
\[
p_{\alpha,n}\left(x,y,t\right):=\int_{0}^{t}\int_{0}^{\infty}p_{\alpha,0}\left(x,\zeta,\tau\right)\varLambda\left(\zeta\right)p_{\alpha,n-1}\left(\zeta,y,t-\tau\right)dud\tau\text{ for }n\geq1.
\]
Then, for every $\left(x,y,t\right)\in\left(0,\infty\right)^{3}$,
$\sum_{n=0}^{\infty}p_{\alpha,n}\left(x,y,t\right)$ converges absolutely
and 
\[
p\left(x,y,t\right):=\sum_{n=0}^{\infty}p_{\alpha,n}\left(x,y,t\right)
\]
is the fundamental solution to (\ref{eq:IVP with x^(alpha) and drift}).

In addition, for every $t>0$,
\[
\sup_{x,y\in(0,\infty)^{2}}\left|\frac{p\left(x,y,t\right)}{p_{\alpha}\left(x,y,t\right)e^{\frac{1}{2}\int_{x}^{y}u^{\beta-\alpha}\varphi\left(u\right)du}}-1\right|\leq e^{\left\Vert \varLambda\right\Vert _{u}t}-1,
\]
and for every $k\in\mathbb{N}$, 
\[
\sup_{x,y\in(0,\infty)^{2}}\left|\frac{p\left(x,y,t\right)-\sum_{j=0}^{k}\,p_{\alpha,j}\left(x,y,t\right)}{p_{\alpha}\left(x,y,t\right)e^{\frac{1}{2}\int_{x}^{y}u^{\beta-\alpha}\varphi\left(u\right)du}}\right|\leq e^{\left\Vert \varLambda\right\Vert _{u}t}\frac{\left\Vert \varLambda\right\Vert _{u}^{k+1}}{\left(k+1\right)!}t^{k+1}.
\]
\end{prop}

\begin{proof}
Following the notations in $\mathsection3$, one easily check that
for every $x>0$,
\[
\theta\left(\phi\left(x\right)\right)=\exp\left(-\frac{1}{2}\int_{0}^{x}u^{\beta-\alpha}\varphi\left(u\right)du\right)\text{ and }V\left(\phi\left(x\right)\right)=\varLambda\left(x\right).
\]
Apparently we want to define, for $\left(x,y,t\right)\in\left(0,\infty\right)^{3}$
and $n\in\mathbb{N}$,
\[
p_{\alpha,n}\left(x,y,t\right):=q_{\nu,n}\left(\phi\left(x\right),\phi\left(y\right),t\right)\frac{\theta\left(\phi\left(x\right)\right)}{\theta\left(\phi\left(y\right)\right)}\phi^{\prime}\left(y\right),
\]
which leads to
\[
p_{\alpha,0}\left(x,y,t\right)=p_{\alpha}\left(x,y,t\right)e^{\frac{1}{2}\int_{x}^{y}u^{\beta-\alpha}\varphi\left(u\right)du}
\]
and for $n\geq1$, by (\ref{eq:recursion n->n+1}), 
\[
\begin{split}p_{\alpha,n}\left(x,y,t\right) & =\int_{0}^{t}\int_{0}^{\infty}q_{\nu}\left(\phi\left(x\right),\xi,t-\tau\right)V\left(\xi\right)q_{\nu,n}\left(\xi,\phi\left(y\right),\tau\right)d\xi d\tau\cdot e^{\frac{1}{2}\int_{x}^{y}u^{\beta-\alpha}\varphi\left(u\right)du}\phi^{\prime}\left(y\right)\\
 & =\int_{0}^{t}\int_{0}^{\infty}q_{\nu}\left(\phi\left(x\right),\phi\left(\zeta\right),t-\tau\right)\phi^{\prime}\left(\zeta\right)e^{\frac{1}{2}\int_{x}^{y}u^{\beta-\alpha}\varphi\left(u\right)du}\\
 & \hspace{2cm}\cdot V\left(\phi\left(\zeta\right)\right)q_{\nu,n-1}\left(\phi\left(\zeta\right),\phi\left(y\right),\tau\right)e^{\frac{1}{2}\int_{x}^{y}u^{\beta-\alpha}\varphi\left(u\right)du}\phi^{\prime}\left(y\right)d\zeta d\tau\\
 & =\int_{0}^{t}\int_{0}^{\infty}p_{\alpha,0}\left(x,\zeta,t-\tau\right)\varLambda\left(\zeta\right)p_{\alpha,n-1}\left(\zeta,y,\tau\right)d\zeta d\tau.
\end{split}
\]
The rest follows immediately from Theorem \ref{thm:main result on p(x,y,t)}
and Corollary \ref{cor:p_approx. ratio estimate}.
\end{proof}

\subsection{Further questions.}

Although we are primarily interested in the near-boundary behaviors
of the solutions and the fundamental solutions, the results we obtained
in the previous sections, e.g., Corollary \ref{cor:p_approx. ratio estimate}
and Proposition \ref{prop:regularity of p(x,y,t) and u_f}, concern
the properties of these functions on the entire domain $\left(0,\infty\right)$.
As a consequence, our arguments rely on the global conditions such
as Condition 1-3 on $a\left(x\right)$ and $b\left(x\right)$, and
our results involve the global bounds of the coefficients such as
$C_{k}^{V}$. This set-up certainly puts strong constraints on the
coefficients that can be treated with our method. Meanwhile, if one
only focuses on the behaviors of the fundamental solutions near 0,
then one would expect that only the near-0 properties of the coefficients
matter. For example, in the example $L=x^{\alpha}\partial_{x}^{2}+x^{\beta}\varphi\left(x\right)\partial_{x}$
discussed above, it is reasonable to believe that, when $x,y$ are
near 0, $p\left(x,y,t\right)$ behaves similarly as if the operator
were $L=x^{\alpha}\partial_{x}^{2}+cx^{\beta}\partial_{x}$ where
$c:=\lim_{x\searrow0}\varphi\left(x\right)$, and the actual formula
of $\varphi$ should not have any impact on the regularity of $u_{f}\left(x,t\right)$
and $p\left(x,y,t\right)$ near 0. Thus, if we could replace, at least
locally near 0, ``$\varphi\left(x\right)$'' by ``$cx^{\beta}$''
in the derivation of Proposition \ref{prop:p(x,y,t) of x^alapha(dx)^2 and a drift},
then the construction of $p\left(x,y,t\right)$ would be more accessible,
and it would even be possible to derive sharp estimates on $p\left(x,y,t\right)$
that are in closed-form expressions. We hope to make such considerations
rigorous in the next step. To proceed from here, we aim to find a
way to ``localize'' the methods and the results from the previous
sections, and to relax the global conditions on $a\left(x\right)$
and $b\left(x\right)$. This idea of localization is again inspired
by \cite{siamwfeq} and based on a probabilistic approach of examining
the recursions of the associated diffusion process in a neighborhood
of 0. We will carry out such a project in the sequel to this paper,
in which we will re-derive the near-0 estimates on the solutions and
the fundamental solutions that only involve local bounds of the coefficients.
In this case more general $a\left(x\right)$ and $b\left(x\right)$
can be treated, including those who do not behave well, such as blowing
up, on $\left(0,\infty\right)$ away from 0. 

\bibliographystyle{plain}
\bibliography{mybib}

\end{document}